\documentclass[sigconf]{acmart}
\usepackage{booktabs} % For formal tables
\usepackage{graphicx}
\usepackage{balance}  % for  \balance command ON LAST PAGE  (only there!)
\usepackage{algorithm}
\usepackage{algorithmic}
\usepackage{graphicx}
\usepackage{adjustbox}
\usepackage{xspace}
\usepackage{hyperref}
\usepackage{amsmath}
\usepackage{amssymb}
\usepackage{color}
\usepackage{xcolor}
\usepackage{tablefootnote} % For table footnote
\usepackage[center]{subfigure} % for \subfloat

\renewcommand\footnotetextcopyrightpermission[1]{} 

\setcopyright{none}
\newtheorem{problem}{Problem}
\newtheorem{definition}{Definition}

\newtheorem{theorem}{Theorem}

\newcommand{\method}{\textsc{Zoom-SVD}\xspace}

\newcommand{\incrementalsvd}{\textsc{Incremental-SVD}\xspace}
\newcommand{\partialsvd}{\textsc{Partial-SVD}\xspace}
\newcommand{\stitchedsvd}{\textsc{Stitched-SVD}\xspace}

\newcommand{\hide}[1]{}

\newcommand*{\QEDB}{\hfill\ensuremath{\Box}}%

\newcommand{\mat}[1]{\mathbf{#1}}
\newcommand{\matt}[1]{\mathbf{#1}^{\text{T}}}
\newcommand{\vect}[1]{\mathbf{#1}}

\newcommand{\Abmat}[1]{\mat{A}^{(#1)}}
\newcommand{\Umat}[1]{\mat{U}_{(#1)}}
\newcommand{\Smat}[1]{\mat{\Sigma}_{(#1)}}
\newcommand{\Vmat}[1]{\mat{V}_{(#1)}}
\newcommand{\Vmatt}[1]{\matt{V}_{(#1)}}
\newcommand{\Vmattt}[1]{\mat{V'}^{\text{T}}_{(#1)}}
\newcommand{\Zmat}[0]{\mat{O}}
\newcommand{\Imat}[0]{\mat{I}}

\newcommand{\Uset}{\mathcal{U}}
\newcommand{\Sset}{\mathcal{S}}
\newcommand{\Vset}{\mathcal{V}}

\newcommand{\subfloat}{\subfigure}

\begin{document}
\title{Zoom-SVD: Fast and Memory Efficient Method for\\
Extracting Key Patterns in an Arbitrary Time Range
}
%\titlenote{Produces the permission block, and
%  copyright information}
%\subtitle{Extended Abstract}
%\subtitlenote{The full version of the author's guide is available as
%  \texttt{acmart.pdf} document}

%\author{Anonymous}

\author{Jun-Gi Jang}
\affiliation{%
  \institution{Seoul National University}
}
\email{elnino4@snu.ac.kr}

\author{Dongjin Choi}
\affiliation{%
  \institution{Seoul National University}
}
\email{skywalker5@snu.ac.kr}

\author{Jinhong Jung}
\affiliation{%
  \institution{Seoul National University}
}
\email{jinhongjung@snu.ac.kr}

\author{U Kang}
\affiliation{%
  \institution{Seoul National University}
}
\email{ukang@snu.ac.kr}

\begin{abstract}
	Given multiple time series data, how can we efficiently find latent patterns in an arbitrary time range?
Singular value decomposition (SVD) is a crucial tool to discover hidden factors in multiple time series data, and has been used in many data mining applications including dimensionality reduction, principal component analysis, recommender systems, etc.
Along with its static version, incremental SVD has been used to deal with multiple semi-infinite time series data and to identify patterns of the data.
However, existing SVD methods for the multiple time series data analysis do not provide functionality for detecting patterns of data in an arbitrary time range: standard SVD requires data for all intervals corresponding to a time range query, and incremental SVD does not consider an arbitrary time range.

In this paper, we propose \method, a fast and memory efficient method for finding latent factors of time series data in an arbitrary time range.
\method incrementally compresses multiple time series data block by block to reduce the space cost in storage phase,
and efficiently computes singular value decomposition (SVD) for a given time range query in query phase by carefully stitching stored SVD results.
Through extensive experiments, we demonstrate that \method is up to $15 \times$ faster, and requires $15 \times$ less space than existing methods.
Our case study shows that \method is useful for capturing past time ranges whose patterns are similar to a query time range.

%past time ranges which have similar patterns to that of a query time range.
% 
\end{abstract}

%\keywords{Time range query; singular value decomposition; multiple time series data}

\maketitle

\section{Introduction}
\label{sec:intro}

Given multiple time series data (e.g., measurements from multiple sensors) and a time range (e.g., 1:00 am - 3:00 am yesterday), how can we efficiently discover latent factors of the time series in the range?
Revealing hidden factors in time series is important for analysis of patterns and tendencies encoded in the time series data. %, and  has been considered as a significant research topic in data mining comnunity~\cite{}.
Singular value decomposition (SVD) effectively finds hidden factors in data, and has been extensively utilized in many data mining applications such as dimensionality reduction~\cite{ravi1998dimensionality}, principal component analysis (PCA)~\cite{jolliffe2002principal,wall2003singular}, data clustering~\cite{simek2004using,osinski2004lingo}, tensor analysis~\cite{Sael201582,JeonPKF15,journals/vldb/JeonPFSK16,conf/icde/JeonJSK16,conf/cikm/ParkJLK16,conf/icde/OhPSK18}, graph mining~\cite{KangTS12,tong2006fast,KangMF11,KangMF14} and recommender systems~\cite{koren2009matrix, conf/bigdataconf/ParkJK17}.
SVD has been also successfully applied to stream mining tasks~\cite{wall2003singular,spiegel2011pattern} in order to analyze time series data.
\begin{figure*}
	\vspace{-4mm}
	 \subfloat{\includegraphics[width=0.80\textwidth]{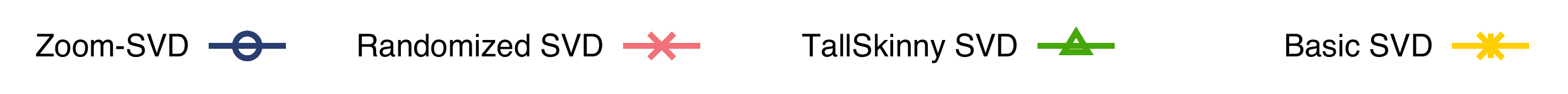}}\vspace{-6mm} \\
	\subfloat[Query time of the Activity dataset]{\includegraphics[width=0.3\textwidth]{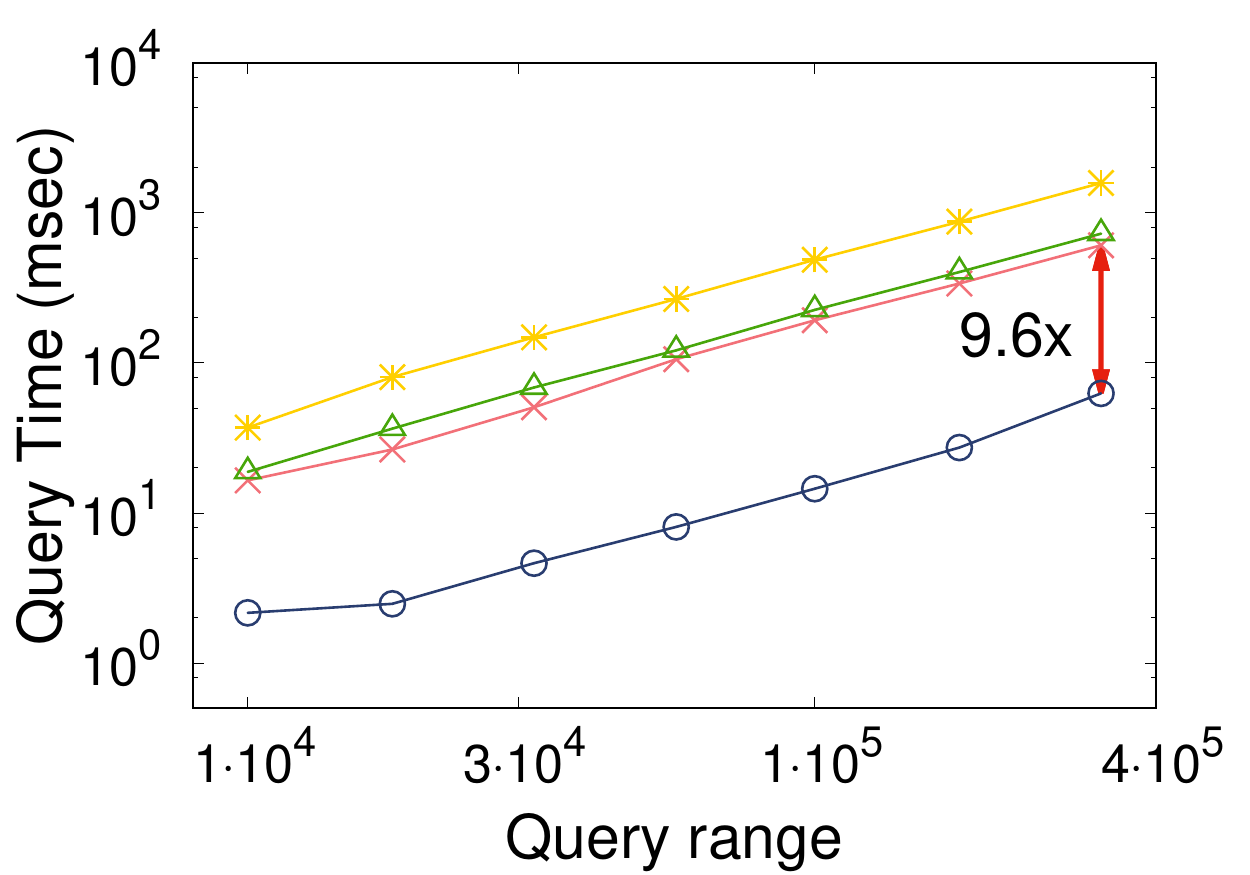}\label{fig:Query_Activity}}
	\subfloat[Query time of the Gas dataset]{\includegraphics[width=0.3\textwidth]{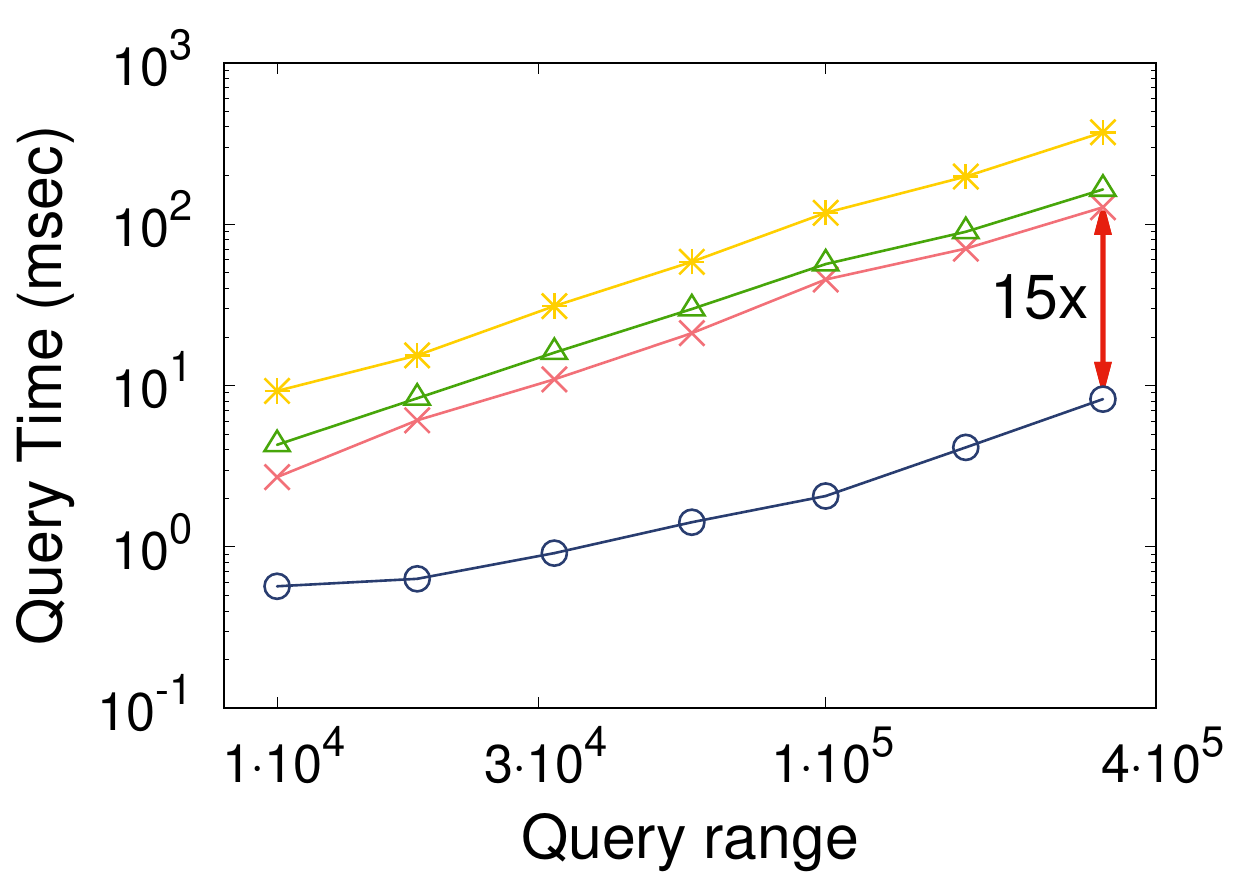}\label{fig:Query_Gas}}
	\subfloat[Query time of the London dataset]{\includegraphics[width=0.3\textwidth]{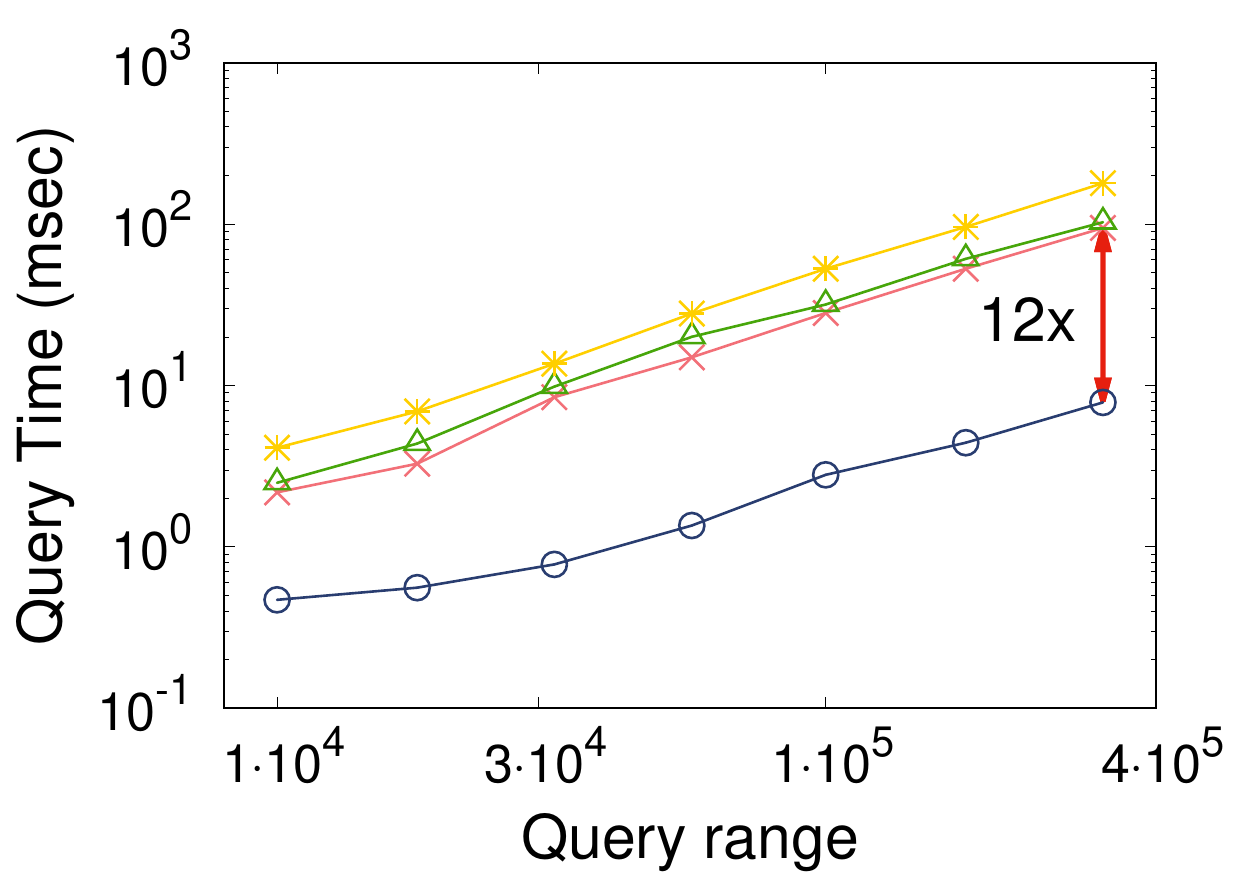}\label{fig:Query_London}} \\
	\caption{
		Query time of \method compared to other SVD methods.
		The starting point $t_s$ and the ending point $t_e$ are arbitrarily chosen, and we increase time range $t_{e} - t_{s}$ from $10^4$ to $3.2 \times 10^5$. % by $5 \times 10^3$.
		(a) The query time of \method is up to $9.6 \times$ faster than that of the second best method in Activity dataset.
		(b) \method is up to $15 \times$ faster than the second best method in the query phase of Gas dataset.
		(c) \method is up to $12 \times$ faster than the second best method in the query phase of London dataset.
		%The rank of stored block SVD affects the query time since the block size is set to the same value $10^3$, and there is little difference between the column dimension of two datasets.
	}
	\label{fig:Runtime}
\end{figure*}
However, methods based on standard SVD~\cite{brand2003fast,ross2008incremental,DBLP:conf/kdd/ZadehMUYPVSSZ16,DBLP:journals/siamrev/HalkoMT11} are not suitable for finding latent factors in an arbitrary time range since the methods have an expensive computational cost, and they have to store all the raw data.
%incrementally updating previous SVD results for a newly arrived data.
This limitation makes it difficult to investigate patterns of a time range in stream environment even if it is important to analyze a specific past event or find recurring patterns in time series~\cite{DBLP:conf/sigmod/PapadimitriouY06}.
A naive approach for a time range query on time series is to store all of the arrived data and apply SVD to the data, but this approach is inefficient since it requires huge storage space, and the computational cost of SVD for a long time range query is expensive.

%we cannot store all stream data and computing resources are restricted.

In this paper, we propose \method (Zoomable SVD), an efficient method for revealing hidden factors of multiple time series in an arbitrary time range.
With \method, users can zoom-in to find patterns in a specific time range of interest, or zoom-out to extract patterns in a wider time range.
\method comprises two phases: storage phase and query phase.
\method considers multiple time series as a set of blocks of a fixed length.
In the storage phase, \method carefully compresses each block using SVD and low-rank approximation to reduce storage cost and incrementally updates the most recent block of a newly arrived data.
In the query phase, \method efficiently computes the SVD results in a given time range based on the compressed blocks.
Through extensive experiments with real-world multiple time series data, we demonstrate the effectiveness and the efficiency of \method compared to other methods as shown in Figure \ref{fig:Runtime}.
The main contributions of this paper are summarized as follows:

\begin{itemize}
	\item {\textbf{Algorithm.} We propose \method, an efficient method for extracting key patterns from multiple time series data in an arbitrary time range. }
	\item {\textbf{Analysis.} We theoretically analyze the time and the space complexities of our proposed method \method.}
%	\item and show that the complexities of \method are less than those of its competitors.}
	\item {\textbf{Experiment.} We present experimental results showing that \method computes time range queries up to $15 \times$ faster, and requires up to $15 \times$ less space than other methods.
	    We also confirm that our proposed method \method provides the best trade-off between efficiency and accuracy.
	}
  %a novel problem called \textit{Time Range Query}, and design
  %\item \textbf{Novel algorithm for computing SVD in an arbitrary time range}. We propose TR-SVD which compress multivariate time series data and efficiently compute SVD of data obtained between two time tick.
  %\item \textbf{Theory}. We guarantees the performance of TR-SVD in terms of time and space. We show that computational cost is lower than that of existing methods.
  %\item \textbf{Experiment}. TR-SVD is up to 127 times faster, and 4 times more efficient than existing methods. In addition, \color{red} {reconstruction well.}
\end{itemize}

%The code of our method and datasets used in this paper are available at http://anonymized.url.
%{The code, datasets and supplementary document of this paper are available at \url{http://datalab.snu.ac.kr/trsvd}.}

The codes and datasets for this paper are available at \url{http://datalab.snu.ac.kr/zoomsvd}.
In the rest of this paper,
we describe the preliminaries and formally define the problem in Section~\ref{sec:prelim}, propose our method \method in Section~\ref{sec:method},
present experimental results in Section~\ref{sec:experiment},
demonstrate the case study in Section~\ref{sec:casestudy},
discuss related works in Section~\ref{sec:related},
and conclude in Section \ref{sec:conclusion}.

\begin{table} [t]
	\centering
	\caption{Symbol description.}
	\label{tab:notation}
	\resizebox{\columnwidth}{!}{%
		\begin{tabular}{cl}
			\toprule
			\textbf{Symbol} & \textbf{Description} \\
			\midrule
			$b$  & Initial block size  \\
			$\xi$ & Threshold for low-rank approximation \\
			$k$ &  Number of singular values \\
			$k_{(i)}$ & Number of singular values in $i$-th block \\
%			$k_{(S:E)}$ & Number of singular values estimated from query phase \\
			$\begin{bmatrix} \mathbf{X} ~;~ \mathbf{Y} \end{bmatrix}$ & Vertical concatenation of two matrices  $\mathbf{X}$ and $\mathbf{Y}$\\
			$\mathbf{A}$ & Raw multiple time series data \\
			$\Abmat{i}$ & i-th block of $\mathbf{A}$\\
			$\Umat{i}$ & Left singular vector matrix of $\Abmat{i}$ \\
			$\Smat{i}$ & Singular value matrix of $\Abmat{i}$ \\
			$\Vmat{i}$ & Right singular vector matrix of $\Abmat{i}$ \\
			%		$\mathbf{U'_{(i)}}$ & Left singular vectors of $\mathbf{A'^{(i)}}$ \\
			%		$\mathbf{{\Sigma}'_{(i)}}$ & Singular values of $\mathbf{A'^{(i)}}$ \\
			%		$\mathbf{V'_{(i)}}$ & Right singular vectors of $\mathbf{A'^{(i)}}$ \\
			$\Umat{S:E}$ & Left singular vector matrix computed in query phase \\
			$\Smat{S:E}$ & Singular value matrix computed in query phase \\
			$\Vmat{S:E}$ & Right singular vector matrix computed in query phase \\
			$\Uset$ & Set of left singular vector matrix $\Umat{i}$\\
			$\Sset$ & Set of singular value matrix $\Smat{i}$\\
			$\Vset$ & Set of right singular vector matrix $\Vmat{i}$\\
			$[t_s, t_e]$ & Time range query\\
			$t_s$ & Starting point of time range query \\
			$t_e$ & Ending point of time range query\\
			$S$ & Index of block matrix corresponding to $t_s$\\
			$E$ & Index of block matrix corresponding to $t_e$\\
%			$r_{S}$ & Number of rows to be eliminated in $\Abmat{S}$ \\
%			$r_{E}$ & Number of rows to be eliminated in $\Abmat{E}$ \\
%			$l$ & the number of block matrices in the time range $[t_s, t_e]$ \\
			\bottomrule
		\end{tabular}
	}
\end{table}

\section{Preliminaries}
\label{sec:prelim}

We describe preliminaries on singular value decomposition (SVD) and incremental SVD (Sections~\ref{subsec:SVD} and \ref{subsec:IncreSVD}).
We then define the problem handled in this paper (Section~\ref{subsec:problem_definition}).
Table~\ref{tab:notation} lists the symbols used in this paper.

\subsection{Singular Value Decomposition (SVD)}
\label{subsec:SVD}
SVD is a decomposition method for finding latent factors in a matrix $\mathbf{A} \in \mathbb{R}^{t \times c}$. % where $t$ is the number of rows, and $c$ is the number of columns.
Suppose the rank of the matrix $\mathbf{A}$ is $r$.
Then, SVD of $\mathbf{A}$ is represented as $\mathbf{A = U{\Sigma}V^T}$ where
$\mathbf{\Sigma}$ is an $r \times r$ diagonal matrix whose diagonal entries are singular values.
The $i$-th singular value $\sigma_i$ is located in $\Sigma_{i,i}$ where $\sigma_1$ $\geq$ $\sigma_2$ $\geq$ $\cdots$ $\geq$ $\sigma_{r}$ $\geq$ $0$.
$\mathbf{U} \in \mathbb{R}^{t \times r}$ is called the left singular vector matrix (or a set of left singular vectors) of $\mathbf{A}$; $\mathbf{U}=\begin{bmatrix} \mathbf{u}_1 \cdots \mathbf{u}_r \end{bmatrix}$ is a column orthogonal matrix where $\mathbf{u}_1$, $\cdots$, $\mathbf{u}_r$ are the eigenvectors of $\mathbf{A}\mathbf{A}^T$.
$\mathbf{V} \in \mathbb{R}^{c \times r}$ is the right singular vector matrix of $\mathbf{A}$; $\mathbf{V}=\begin{bmatrix} \mathbf{v}_1 \cdots \mathbf{v}_r \end{bmatrix}$ is a column orthogonal matrix where $\mathbf{v}_1$, $\cdots$, $\mathbf{v}_r$ are the eigenvectors of $\mathbf{A}^T\mathbf{A}$.
Note that the singular vectors in $\mathbf{U}$ and $\mathbf{V}$ are used as hidden factors to analyze the data matrix $\mathbf{A}$.

\textbf{Low-rank approximation.}
Low-rank approximation effectively approximates the original data matrix based on SVD.
The key idea of the low-rank approximation is to keep top-$k$ highest singular values and corresponding singular vectors where $k$ is a number smaller than the rank $r$ of the original matrix.
The low-rank approximation of $\mathbf{A}$ is represented as follows:
\begin{equation*}
	\mathbf{A} \simeq \mathbf{U}_{k}\mathbf{\Sigma}_{k}\mathbf{V}_{k}^{\mathbf{T}} \triangleq \hat{\mathbf{A}} \text{ s.t. } k<r
\end{equation*}
\noindent where the reconstruction data $\hat{\mathbf{A}}$ is the low-rank approximation of $\mathbf{A}$, $\mathbf{U}_k = \begin{bmatrix} \mathbf{u}_1,  \cdots, \mathbf{u}_k \end{bmatrix}$, $\mathbf{V}_k = \begin{bmatrix} \mathbf{v}_1, \cdots,  \mathbf{v}_k \end{bmatrix}$, and $\Sigma_{k} = diag(\sigma_{1}, \cdots, \sigma_{k})$.
The error of the low-rank approximation is represented as follows:
\begin{equation*}
	%\label{SVDerror}
	\lVert \mathbf{A} - \hat{\mathbf{A}} \rVert^2_{F} = \lVert \sum_{i=k+1}^{r} \mathbf{\sigma}_i\mathbf{u}_i \mathbf{v}^T_i \rVert^{2}_{F} = \sum_{i=k+1}^{r} \mathbf{\sigma}^{2}_{i}
\end{equation*}
\noindent where $\lVert \cdot \rVert_{F}$ is the Frobenius norm of a matrix, and $r$ is the rank of the original matrix.
The parameter $k$ for low-rank approximation is determined by the following equation:
\begin{equation}
\label{eqn:threshold}
	k^{*} =\underset{1 \leq k \leq r} {\arg \min}  f(k) = \frac{ \sum_{i=1}^{k}\sigma^{2}_{i} }{\sum_{i=1}^{r}\sigma^{2}_{i}} \text{, s.t. } f(k)\geq \xi
%	k^{*} = \min_{1 \leq k \leq r} k \text{ s.t. } \frac{ \sum_{i=1}^{k}\sigma^{2}_{i} }{\sum_{i=1}^{r}\sigma^{2}_{i}} \geq \xi
\end{equation}
\noindent where $\xi$ is a threshold between $0$ and $1$.
%In this paper, we call the aforementioned traditional SVD and low-rank approximation approach as Naive SVD.

\begin{figure*} [t]
	\centering
	\vspace{-4mm}
	\includegraphics [width=0.8\textwidth] {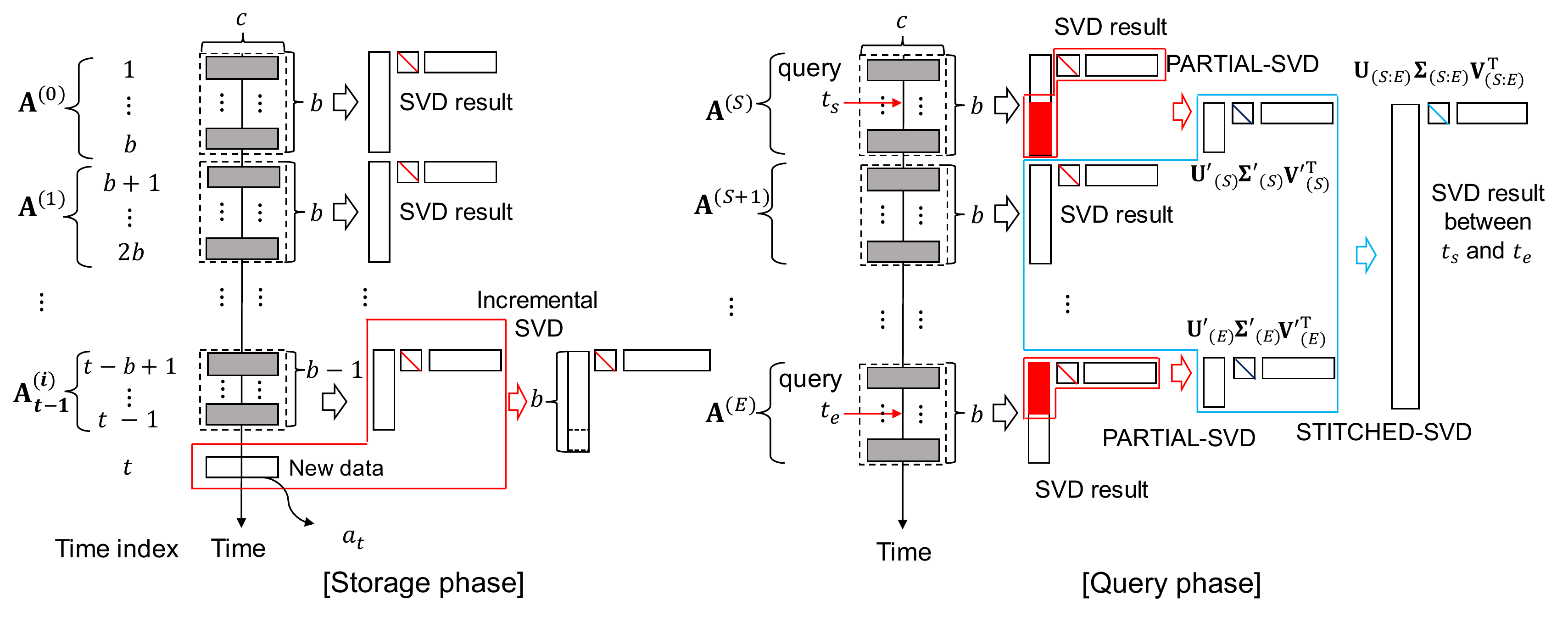}
%	\vspace{-6mm}
	\caption{Overview of \method.
		In the storage phase, \method stores the SVD results for length-$b$ blocks of multiple time series data, updates the SVD result of the most recent block of a newly arrived data, and discards previously arrived data.
		In the query phase, \method computes the SVD result in a given time range $[t_s,t_e]$ based on the small SVD results.
		The query phase exploits \partialsvd (Section~\ref{subsec:partialSVD}) and \stitchedsvd (Section~\ref{subsec:block}) modules to process the time range query.
	}
	\label{fig:proposedmethod}
\end{figure*}

\subsection{Incremental SVD}
\label{subsec:IncreSVD}
Incremental SVD dynamically calculates the SVD result of a matrix with newly arrived data rows.
Suppose that we have the SVD result $\mathbf{U}_{t}$, $\mathbf{\Sigma}_{t}$, and $\mathbf{V}^T_{t}$ of a data matrix $\mathbf{A}_{t}\in \mathbb{R}^{n \times c}$ at time $t$.
When an $m \times c$ matrix $\mathbf{A}_{t+1}$  arrives at time $t+1$, the purpose of incremental SVD is to efficiently obtain the SVD result of $\begin{bmatrix}\mathbf{A}_{t} ~;~ \mathbf{A}_{t+1} \end{bmatrix}$ based on the previous result $\mathbf{U}_{t}$, $\mathbf{\Sigma}_{t}$, and $\mathbf{V}^T_{t}$. Note that $\begin{bmatrix}\mathbf{X} ~;~ \mathbf{Y} \end{bmatrix}$ denotes the vertical concatenation of two matrices $\mathbf{X}$ and $\mathbf{Y}$.
Incremental SVD is used to analyze patterns in time series data~\cite{sarwar2002incremental}, and several efficient methods for incremental SVD were proposed~\cite{brand2003fast,ross2008incremental}.
%(see supplementary material~\cite{Supplementary} for the algorithm of incremental SVD).
{This incremental SVD technique is exploited in our method to incrementally compress and store the data (see Algorithm~\ref{alg:storage_phase} in Section~\ref{sec:storing}).}

\subsection{Problem Definition}
\label{subsec:problem_definition}
%The problem definition of this paper is represented as follows:
%The time range query problem is defined as follows:
We formally define the time range query problem as follows:
\begin{problem}%[Time Range Query on Multiple Time Series Data]
	\textnormal{\textsc{(Time Range Query on Multiple Time Series)}}
	\label{problem:time_range_query}
	%SVD result in an arbitrary time range between $t_s$ and $t_e$.
	%The time range query problem is to obtain the SVD result in an arbitrary time range between $t_s$ and $t_e$.
	\begin{itemize}
		\item{\textbf{Given:}
a time range $[t_s,t_e]$, and multiple time series data represented by a matrix $\mathbf{A} \in \mathbb{R}^{t \times c}$ where $t$ is the length of the time dimension, and $c$ is the number of the time series, }	
		\item{\textbf{Find:} the SVD result of the sub-matrix of $\mathbf{A}$ in the time range quickly, without storing all of $\mathbf{A}$. The SVD result includes $\mathbf{U} \in \mathbb{R}^{(t_e - t_s+1) \times k}$, $\mathbf{\Sigma} \in \mathbb{R}^{k \times k}$, and $\mathbf{V} \in \mathbb{R}^{c \times k}$ where $k$ is the rank of the sub-matrix.}
	\end{itemize}
	 %\hspace{3mm},
 % in the range.
%	
%	 Matrices $\mathbf{U_{(t_e-t_s) \times k}, S_{k \times k}}$, and $\mathbf{V_{k \times c}}$ such that  $\mathbf{A_{(t_e-t_s) \times c}}$ $=$  $\mathbf{U_{(t_e-t_s) \times k}, S_{k \times k}, V^T_{k \times c}}$, where $k$ is rank of $\mathbf{A_{(t_e-t_s) \times c}}$.
\end{problem}

Applying the standard SVD or incremental SVD for the time range query is impractical for the following reasons.
Standard SVD needs to extract the sub-matrix corresponding to the time range before performing decomposition.
%Several papers deal with a skinny and tall matrix.
Iwen et al.~\cite{DBLP:journals/siammax/IwenO16} proposed a hierarchical and distributed approach for computing $\mat{\Sigma}$ and $\mat{V}$ except for $\mat{U}$ of a whole matrix $\mat{A}$.
%{\color{blue}Iwen et al.~\cite{DBLP:journals/siammax/IwenO16} concentrate on hierarchically computing singular values matrix $\mathbf{\Sigma}$ and right singular vector matrix $\mathbf{V}$ using the SVD results of block matrices.}
%split matrix $\mathbf{A}$ into several block matrices, compute SVD of each block matrix, and then compute singular values matrix $\mathbf{\Sigma}$ and right singular vector matrix $\mathbf{V}$ using the SVD results of each block matrix.
{Zadeh et al.~\cite{DBLP:conf/kdd/ZadehMUYPVSSZ16} introduce Tall and Skinny SVD which obtains $\mathbf{\Sigma}$ and $\mathbf{V}$ by computing eigen-decomposition of $\mathbf{{A}^TA}$, and then computes $\mathbf{U}$ using $\mathbf{\Sigma}$, $\mathbf{V}$, and $\mathbf{A}$.
Halko et al.~\cite{DBLP:journals/siamrev/HalkoMT11} propose Randomized SVD which computes SVD of $\mathbf{A}$ using randomized approximation techniques.}
However, such methods are inefficient because they need to compute SVDs from scratch for multiple overlapping queries.
Furthermore, those methods need to keep the entire time series data $\mathbf{A}$, which is practically infeasible in many streaming applications.
Incremental SVD considers updates only on newly added data, and thus cannot perform SVD on a specific time range.
%\end{itemize}

   To address these limitations, we propose an efficient method for the time range query in Section~\ref{sec:method}.

\section{Proposed Method}
\label{sec:method}

{
We propose \method, a fast and space-efficient method for extracting key patterns from multiple time series data in an arbitrary time range.}
We first give an overview of \method in Section~\ref{subsec:overview}.
We	 describe details of \method in Sections~\ref{sec:storing} and~\ref{sec:query}.
Finally, we analyze \method's time and space complexities in Section~\ref{subsec:Theo}.

\subsection{\textbf{Overview}}
\label{subsec:overview}
{\method efficiently extracts key patterns from multiple time series data in an arbitrary time range using SVD.}
%computes SVD of multiple time series data in an arbitrary time range.
The main challenges for the time range query problem (Problem~\ref{problem:time_range_query}) are as follows:
\begin{enumerate}
\item \textbf{Minimize the space cost.}
	The amount of multiple time series data increases over time.
	How can we reduce the space while supporting time range queries?
%\item \textbf{Minimize the size of intermidiate data.} Intermidiate data which reconstruct stored data affect the performance of computing SVD. How can we remove the intermidiate data?
\item \textbf{Minimize the time cost.}
	How can we quickly compute SVD of multiple time series data in an arbitrary time range?
\end{enumerate}

%We propose the following ideas to solve the above challenges.
We address the above challenges with the following ideas:
%which are describe in later sections.
%Also, Figure \ref{fig:proposedmethod} summarizes the main ideas of our approach.

%\begin{enumerate}
%\item \textbf{Compress multiple time series data.} We continuously produce the SVD results of matrices with block size $b$ by updating the time series data incrementally (\textbf{Section~\ref{sec:storing}}).
%\item \textbf{Minimize the number of matrix multiplication.} Partial SVD and Integrated SVD avoid reconstructing the SVD results to original data (\textbf{Section~\ref{subsec:partialSVD}} and \textbf{Section~\ref{subsec:IntegratedSVD}}). Also, we optimize the Integrated SVD by reducing redundant computation (\textbf{Section~\ref{subsec:block}}).
%\end{enumerate}

\begin{enumerate}
\item \textbf{Compress multiple time series data (\textbf{Section~\ref{sec:storing}}).}
{\method compresses the raw data using incremental SVD, and discards the raw data in the storage phase.}
%When the current block size becomes $b$, we store the SVD of the current block in a storage; the next time series input is added to a new SVD result.

\item \textbf{Avoid reconstructing the original data from the SVD results computed in the storage phase (\textbf{Sections~\ref{subsec:partialSVD} and~\ref{subsec:block}}).}
%Reconstruction data is obtained by multiplying $\mat{U}$,$\mat{\Sigma}$, and $\matt{V}$ which are the SVD result of $\mat{A}$.
We propose \partialsvd and \stitchedsvd to compute SVD without reconstructing the original data corresponding to the query time range.

\item \textbf{Optimize the computational time of \stitchedsvd  (\textbf{Section~\ref{subsec:block}}).} We optimize the performance of \stitchedsvd by reducing numerical computations using a block matrix structure.
\end{enumerate}

\method comprises two phases: storage phase and query phase.
In the storage phase (Algorithm~\ref{alg:storage_phase}), \method stores the SVD results corresponding to length-$b$ blocks in the time series data in order to support time range queries as shown in Figure~\ref{fig:proposedmethod}.
When a new data arrives, \method incrementally updates the SVD result with the newly arrived data, block by block.
In the query phase (Algorithms~\ref{alg:opt_query_phase}), \method returns the SVD result for a given time range $[t_s, t_e]$.
The query phase utilizes our proposed \partialsvd and \stitchedsvd modules to process the time range query.
Partial SVD (Algorithm~\ref{alg:partial_svd}) manipulates the SVD result containing $t_s$ (or $t_e$) to match the query time range as shown in Figure~\ref{fig:proposedmethod}.
\stitchedsvd (Algorithms~\ref{alg:opt_query_phase}) efficiently computes the SVD result between $t_s$ and $t_e$ by stitching the SVD results for blocks in the time range.

%Section \ref{sec:storing} shows how to store multiple time series data.
%In Sections \ref{sec:query} We describe how we compute SVD using stored data given an arbitrary time range, and provide the theoretical analysis in terms of time and space.
%Last, we describe the  in \ref{sec:multipleblock}.

%Figure \ref{fig:proposedmethod} illustrate how to deal with the problem \ref{pd} in this paper.
\subsection{Storage Phase of \method}
\label{sec:storing}
%In the storage phase, the first challenge for efficient \method is how to minimize the space cost for multiple time series data.
{Given multiple time series stream $\mathbf{A}$, the objective of the storage phase is to incrementally compress the input data and discard the original input data $\mathbf{A}$ to achieve space efficiency.
A naive incremental SVD would update one large SVD result when the data are newly added.
%However, this approach is inefficient because the time to process newly added data increases over time.
However, this approach is impractical because the processing cost for the newly added data increases over time. % $t$ of $\mat{A}$.
Also, the naive incremental SVD does not support a time range query quickly in the query phase because it manipulates the large SVD result stored for the total time regardless of the query time range.
%we manipulate the stored SVD result for the total time length regardless of time range query.

The storage phase of \method (Algorithm~\ref{alg:storage_phase}) is designed to efficiently process newly added data and quickly support time range queries.
Given multiple time series data $\mathbf{A}$, the storage phase of \method (Algorithm~\ref{alg:storage_phase}) incrementally compresses the input data block by block using incremental SVD, and discards the original input data $\mathbf{A}$ to reduce space cost.}
%Let the multiple time series data be continuously growing $t \times c$ matrix $\mathbf{A} \in \mathbb{R}^{t \times c}$ when one new vector come in at each time tick.
Assume the multiple time series data are represented by a matrix $\mathbf{A} \in \mathbb{R}^{t \times c}$ where $t$ is time length, and $c$ is the number of time series (e.g., sensors).
We conceptually divide the matrix $\mathbf{A}$ into length-$b$ blocks represented by $\Abmat{i} \in \mathbb{R}^{b \times c}$ as shown in Figure~\ref{fig:proposedmethod}.
We then store the low-rank approximation result of each block matrix $\Abmat{i}$, where we exploit an incremental SVD method in the process.
We formally define the block matrix $\Abmat{i}$ in Definition~\ref{def:block_matrix}.
\begin{definition}[Block matrix $\Abmat{i}$]
	\label{def:block_matrix}
		Suppose a multivariate time series is $\mat{A} = {\begin{bmatrix}
			\vect{a}_1; \vect{a}_2; \cdots; \vect{a}_t
		\end{bmatrix}} $ where $\vect{a}_{j} \in \mathbb{R}^{1 \times c}$ is the $j$-th row vector of $\mat{A}$, and $[;]$ denotes the vertical concatenation of vectors.
		The $i$-th block matrix $\Abmat{i}$ is then represented as follows:
		\begin{equation*}
			\Abmat{i} \equiv {\begin{bmatrix}
 			\vect{a}_{b \times i+1}; \vect{a}_{b \times i+2}; \cdots; \vect{a}_{b\times i+ b}
			\end{bmatrix}}	
		\end{equation*}
		where $b$ is a block size.
		In addition, $\Abmat{i}_{t-1}$ denotes the $i$-th block matrix at time $t-1$ where $i$ indicates the index of the most recent block as shown in Figure~\ref{fig:proposedmethod}.
		Note that the number of rows in $\Abmat{i}_{t-1}$ is less than or equal to $b$. \QEDB
\end{definition}

The computed SVD result $\Umat{i}$, $\Smat{i}$, and $\Vmat{i}$ of each block matrix $\Abmat{i}$ are stored as follows. %in the sets $\Uset$, $\Sset$, and $\Vset$, respectively.

\begin{definition}[Sets of SVD results $\Uset$, $\Sset$, and $\Vset$]
	\label{def:sets}
The sets $\Uset$, $\Sset$, and $\Vset$ store the SVD results
$\Umat{i}$, $\Smat{i}$, and $\Vmat{i}$ for all $i$, respectively. \QEDB
\end{definition}

Note that the original time series data are discarded, and we store only the SVD results which occupy less space than the original data.
The SVD results for block matrices are used in the query phase (Algorithm~\ref{alg:opt_query_phase}).
Now we are ready to describe the details of the storage phase.
%by using an incremental SVD method (Algorithm~\ref{alg:IncreSVD}).

%However, this procedure is inefficient because we have to wait until collection of input data becomes block matrix.
%We utilize an incremental SVD.
%Assume that we have the SVD results $\Umat{i}$, $\Smat{i}$, and $\Vmat{i}$ of block matrices $\mat{A}^{(i)}$ for $i = 0, 1, 2, \cdots, k-1$ where $i$ is the index of block matrices.
%We incrementally update the SVD result $\mathbf{U_{(k)}\Sigma_{(k)}V^T_{(k)}}$ using Algorithm~\ref{alg:IncreSVD} whenever multiple time series data come in.
%We store the SVD result when this result corresponds to the SVD result of block matrix $\mathbf{A^{(k)}}$.
%After that, we repeatedly perform the same operation as getting the SVD result of block matrix $\mathbf{A^{(k)}}$ to obtain the SVD results of next block matrices $\mathbf{A^{(i)}}$ for $i = k+1, k+2, \cdots$.
%We only store the SVD results, and we discard the matrix $\mathbf{A}$.

%\subsubsection{\textbf{Compress multiple time series data block by block}}
%\label{subsec:Collection}
 \begin{algorithm} [t]
	\caption{Storage phase of \method}\label{alg:storage_phase}
	\begin{algorithmic} [1]
%		\footnotesize
		\small
		\algsetup{linenosize=\small}

		\renewcommand{\algorithmicrequire}{\textbf{Input:}}
		\renewcommand{\algorithmicensure}{\textbf{Output:}}
		    \REQUIRE multiple time series data \textbf{$\mathbf{A}$}
		    \ENSURE set $\Uset$ of left singular vector matrix $\Umat{i}$, set $\Sset$ of singular value matrix $\Smat{i}$, and set $\Vset$ of right singular vector matrix $\Vmat{i}$
		\renewcommand{\algorithmicrequire}{\textbf{Parameters:}}
		\REQUIRE block size $b$\\
		%\STATE \textbf{Initialize:} $i = 0$, Compute $\mathbf{a_{1}}$ $\overset{\mathbf{SVD}}{\leftarrow}$  $\mathbf{U_{0}\Sigma_{0}{V}^T_{0}}$, $\mathcal{U} \cup \{\mathbf{U_{0}}\}$, $\mathcal{S} \cup \{\mathbf{\Sigma_{0}}\}$, $\mathcal{V} \cup \{\mathbf{V_{0}}\}$
		%\STATE \textbf{Initialization:}
		\STATE \textbf{Initialize:} $\Umat{0,1}$, $\Smat{0,1}$, and $\Vmat{0,1}$ $\gets$ SVD result of $\vect{a}_{1}$, block index $i \leftarrow 0$, and time index $t \leftarrow 2$
		%    \STATE Set $i \leftarrow 0$, $\Uset \leftarrow \{\Umat{0}\}$, $\Sset \leftarrow \{\Smat{0}\}$, and $\Vset \leftarrow \{\Vmat{0}^T\}$
		\WHILE {$\vect{a}_{t}$ is newly arriving} \label{alg:storage_phase:new_data}
		\IF {$\Umat{i,t-1}$, $\Smat{i,t-1}$, $\Vmat{i,t-1}$ do not exist} \label{alg:storage_phase:check_block:exist}
			\STATE {$\Umat{i,t}$, $\Smat{i,t}$, $\Vmat{i,t}$ $\gets$ SVD result of $\vect{a}_{t}$} \label{alg:storage_phase:new_generate}
		\ELSE
			\STATE{$\Umat{i,t}$, $\Smat{i,t}$, $\Vmat{i,t}$ $\gets$ \incrementalsvd($\Umat{i,t-1}$, $\Smat{i,t-1}$, $\Vmat{i,t-1}$, $\vect{a}_{t}$) \label{alg:line:update}}
		\ENDIF
		\IF {the number of rows of $\Umat{i,t}$ is equal to $b$} \label{alg:storage_phase:check_block:start}
		%\STATE Compute the low-rank approximation of $\Umat{i}$, $\Smat{i}$, and $\Vmat{i}$
			\STATE \footnotesize {$\Uset \gets \Uset \cup \{\Umat{i,t}\}$}, {$\Sset \gets \Sset \cup \{\Smat{i,t}\}$}, {$\Vset \gets \Vset \cup \{\Vmat{i,t}\}$} \label{alg:storage_phase:store}
			\STATE {$i \gets i+1$} \label{alg:storage_phase:inc_i}
		\ENDIF  \label{alg:storage_phase:check_block:end}	
		%\STATE {Compute $\mathbf{a_{t}} = \mathbf{U_{(i)}{\Sigma}_{(i)}{V_{(i)}}^{T}}$ \{low-rank approx. SVD\}} \label{alg:storage_phase:new_generate}
%TODO ref equation about low-rank
		%\STATE {compute $\Umat{i,t}$, $\Smat{i,t}$, and $\Vmat{i,t}$ based on $\Umat{i,t-1}$, $\Smat{i,t-1}$, $\Vmat{i,t-1}$, and $\vect{a}_{t}$ using \incrementalsvd (Algorithm~\ref{alg:IncreSVD})} \label{alg:line:update}
		\STATE {$t \gets t + 1$}
		\ENDWHILE
		%    \RETURN $\Uset$, $\Sset$, and $\Vset$
	\end{algorithmic}
\end{algorithm}

The storage phase (Algorithm~\ref{alg:storage_phase}) compresses the multiple time series data block by block using {incremental SVD} to support time range queries.
When new multiple time series data $\vect{a}_{t} \in \mathbb{R}^{1 \times c}$ are given at time $t$ (line~\ref{alg:storage_phase:new_data}), we generate the new SVD result of $\vect{a}_{t}$ for the next block matrix $\Abmat{i}_{t}$ if the SVD result are stored in $\Uset$, $\Sset$, and $\Vset$ at time $t-1$ (lines~\ref{alg:storage_phase:check_block:exist} and ~\ref{alg:storage_phase:new_generate}).
If not, we have the SVD result $\Umat{i,t-1}$, $\Smat{i,t-1}$, and $\Vmat{i,t-1}$ of the most recent block matrix $\Abmat{i}_{t-1}$ which is the $i$-th block matrix at time $t-1$.
{Assume that we have the SVD result $\Umat{i,t-1}$, $\Smat{i,t-1}$, and $\Vmat{i,t-1}$ of $\Abmat{i}_{t-1}$ (i.e., the block matrix from time $t-b+1$ to $t-1$ as seen in Figure~\ref{fig:proposedmethod}).
%Note that $k_{t-1}$ is the rank of the SVD result computed at time $t-1$.
We then update the SVD result into $\Umat{i,t}$, $\Smat{i,t}$, and $\Vmat{i,t}$  for the new data $\vect{a}_{t}$ using an incremental SVD method (line~\ref{alg:line:update}).
If the number of rows of $\Umat{i,t}$ is $b$, we put the SVD result $\Umat{i,t}$, $\Smat{i,t}$, and $\Vmat{i,t}$ into $\Uset$, $\Sset$, and $\Vset$, respectively (lines \ref{alg:storage_phase:check_block:start}$\sim$\ref{alg:storage_phase:check_block:end}).}
Equations~\eqref{eqn:Collection} and \eqref{eqn:Collection2} represent the details of how to update the SVD result of $\Abmat{i}_{t-1}$ for the new incoming data $\vect{a}_{t}$, when $\Abmat{i}_{t-1}$ contains $b-1$ rows.
{$\Abmat{i}_{t}$ is represented by $\vect{a}_{t}$ and the SVD result of $\Abmat{i}_{t-1}$ in Equation~\eqref{eqn:Collection}:}
\begin{align}
\footnotesize
	\begin{split}
	\label{eqn:Collection}
		\Abmat{i}_{t} = \begin{bmatrix}
			\Abmat{i}_{t-1} \\
			\vect{a}_{t}
		\end{bmatrix} &\simeq
		\begin{bmatrix}
			\Umat{i,t-1}\Smat{i,t-1}\Vmatt{i,t-1} \\
			\vect{a}_{t}
		\end{bmatrix} \\
		&= \begin{bmatrix}
			\Umat{i,t-1} & \Zmat_{(b-1) \times 1} \\
			\Zmat_{1 \times k_{t-1}} & \Imat_{1 \times 1}	
		\end{bmatrix}
		\begin{bmatrix}
			\Smat{i,t-1}\Vmatt{i,t-1} \\
			\vect{a}_{t}
		\end{bmatrix}
	\end{split}
\end{align}
\noindent where $\Zmat_{x \times y}$ is an $x \times y$ zero matrix, and $\Imat_{x \times x}$ is an $x \times x$ identity matrix.
We then perform SVD to decompose {\footnotesize$\begin{bmatrix}
			\Smat{i,t-1}\Vmatt{i,t-1} \\
			\vect{a}_{t}
		\end{bmatrix}$} $\in \mathbb{R}^{(k_{t-1} + 1) \times c}$ into $\mat{\tilde{U}}\mat{\tilde{\Sigma}}\matt{\tilde{V}}$:

\vspace{-2mm}

\begin{align}
\footnotesize
	\label{eqn:Collection2}
	\begin{split}
	&\begin{bmatrix}
		\Umat{i,t-1} & \Zmat_{(b-1) \times 1} \\
		\Zmat_{1 \times k_{t-1}} & \Imat_{1 \times 1}	
	\end{bmatrix}
	\begin{bmatrix}
		\Smat{i,t-1}\Vmatt{i,t-1} \\
		\vect{a}_{t}
	\end{bmatrix} \\
	&=
	\begin{bmatrix}
		\Umat{i,t-1} & \Zmat_{(b-1) \times 1} \\
		\Zmat_{1 \times k_{t-1}} & \Imat_{1 \times 1}	
	\end{bmatrix}
	\mat{\tilde{U}}\mat{\tilde{\Sigma}}\matt{\tilde{V}}
	\triangleq \Umat{i,t}\Smat{i,t}\Vmatt{i,t}
	\end{split}
\end{align}
\noindent where $\Umat{i,t}=$ {\footnotesize$\begin{bmatrix}
		\Umat{i,t-1} & \Zmat_{(b-1) \times 1} \\
		\Zmat_{1 \times k_{t-1}} & \Imat_{1 \times 1}	
	\end{bmatrix}$}
	$\mat{\tilde{U}}$, $\Smat{i,t} = \mat{\tilde{\Sigma}}$, and $\Vmatt{i,t} = \matt{\tilde{V}}$.
Note that $\Umat{i,t}$ is a column orthogonal matrix since it is the product of two orthogonal matrices.
$\Vmat{i,t}$ is also column orthogonal, and $\Smat{i,t}$ is a diagonal matrix whose diagonal entries are sorted in the descending order.
Hence, $\Umat{i,t}$, $\Smat{i,t}$, and $\Vmat{i,t}$ are considered as the SVD result of $\Abmat{i}_{t}$ by the definition of SVD~\cite{trefethen1997numerical}.
The time index $t$ can be omitted as in $\Umat{i}$, $\Smat{i}$, $\Vmat{i}$, and $\Abmat{i}$, as described in Definitions~\ref{def:block_matrix} and \ref{def:sets}, if the number of rows of $\Umat{i, t}$ is $b$.
\begin{figure}[t]
	\centering
	\includegraphics[width=0.98\linewidth]{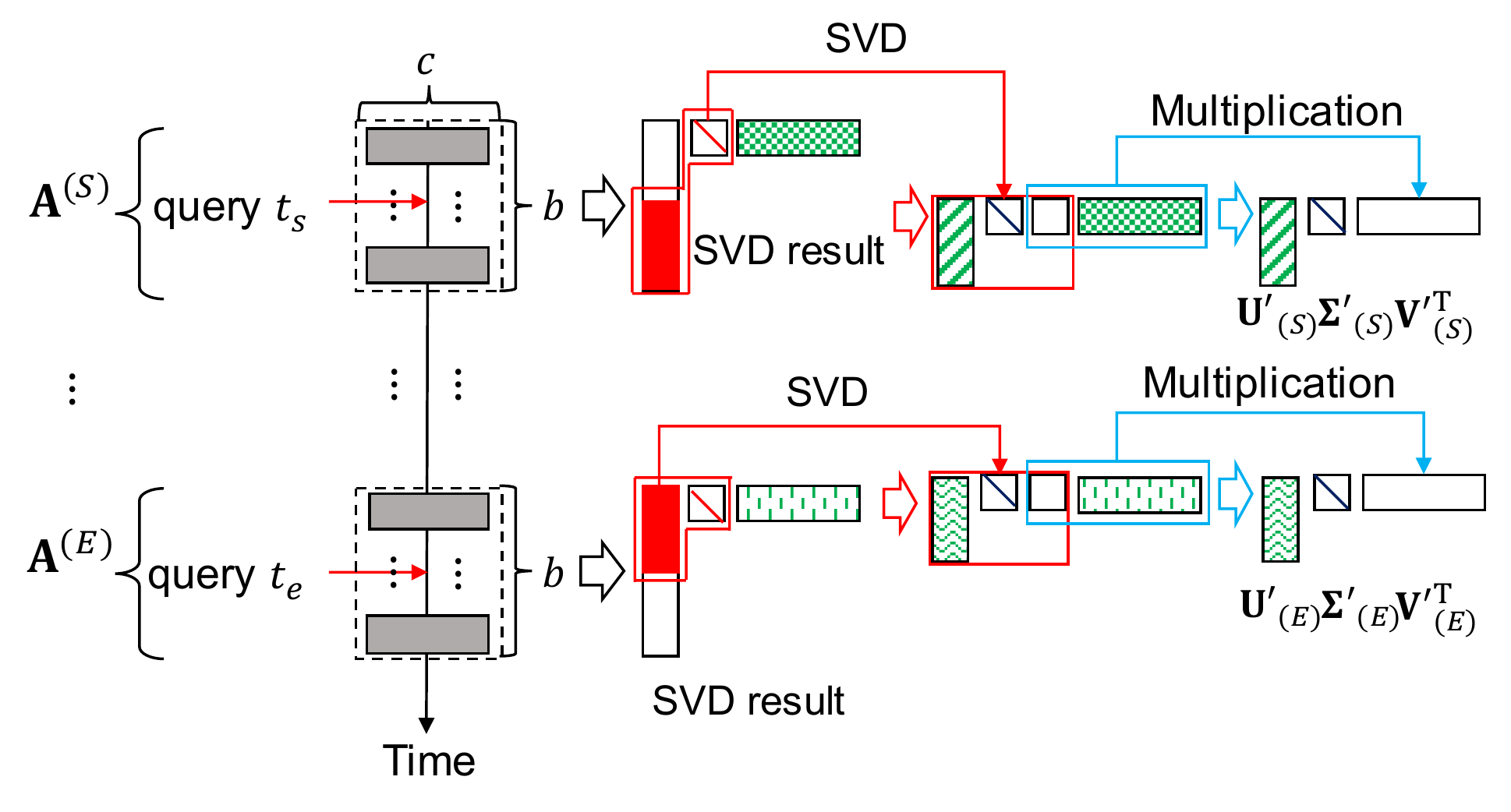}
	\caption{
		{
		Example of \partialsvd. %$\partialsvd$.
		Given a time range query $[t_{s},t_{e}]$, we remove rows of $\Umat{S}$ and $\Umat{E}$ which are out of the query time range. % using the elimination matrices $\mat{X}_{s}$ and $\mat{X}_{e}$.
{\partialsvd exploits SVD on the filtered matrices for left singular vector matrices and singular value matrices within the query time range (red-colored boxes).}
		Then \partialsvd computes right singular vector matrices of the output SVD results by multiplying relevant matrices (blue-colored boxes). % based on the intermediate matrices (blue-colored boxes).
		}
	}
	\label{fig:partial_svd}
\end{figure}

\subsection{Query Phase of \method}
\label{sec:query}
Given the starting point $t_s$ and the ending point $t_e$ of a time range query, the goal of the query phase of \method is to obtain the SVD result from $t_s$ to $t_e$.
A naive approach would reconstruct the time series data from the SVD results of the block matrices ranged between $t_s$ and $t_e$, and perform SVD on the reconstructed data in the range.
However, this approach requires heavy computations especially for a long time range query, and thus is not appropriate for serving time range queries quickly.

We propose two sub-modules, \partialsvd and \stitchedsvd, which are used in the query phase of our proposed method (Algorithm~\ref{alg:opt_query_phase}) to efficiently process time range queries by avoiding reconstruction of the raw data.
%In addition, we optimize the \stitchedsvd by reducing numerical computations using a block structure.
%The query phase returns the SVD result in a given range $[t_s, t_e]$.
Let $S$ be the index of the block matrix including $t_s$, and $E$ be the index of the block matrix including $t_e$.
\partialsvd (Algorithm~\ref{alg:partial_svd}) adjusts the time range of the SVD results for $\Abmat{S}$ and $\Abmat{E}$ as seen in the red-colored boxes of Figure~\ref{fig:partial_svd} (line~\ref{alg:line:part} of Algorithm~\ref{alg:opt_query_phase}).
\stitchedsvd combines the SVD results of \partialsvd and those of block matrices from $\Abmat{S+1}$ to $\Abmat{E-1}$ (lines ~\ref{alg:opt_query_phase:vertical} to ~\ref{alg:opt_query_phase:block_diag} in Algorithm~\ref{alg:opt_query_phase}).
%Note that \stitchedsvd generates the same result as decomposing the concatenation of the block matrices from $t_s$ and $t_e$ without reconstructing the time series data.
%in the range from the SVD results.
We describe the details of \partialsvd and \stitchedsvd in Sections~\ref{subsec:partialSVD} and \ref{subsec:block}, respectively.

\begin{algorithm} [t]
	\caption{Query phase of \method}\label{alg:opt_query_phase}
	\begin{algorithmic} [1]
		\algsetup{linenosize=\small}
		\small
%		\footnotesize
		\renewcommand{\algorithmicrequire}{\textbf{Input:}}
		\renewcommand{\algorithmicensure}{\textbf{Output:}}
		\REQUIRE sets $\Uset$, $\Sset$, and $\Vset$ of block SVD results, starting point $t_{s}$, and ending point $t_{e}$
		\ENSURE SVD result $\Umat{S:E}$, $\Smat{S:E}$, and $\Vmat{S:E}$ in $[t_s, t_e]$
		\renewcommand{\algorithmicrequire}{\textbf{Parameters:}}
		\STATE $\Umat{S}'$, $\Smat{S}'$, $\Vmat{S}'$, $\Umat{E}'$, $\Smat{E}'$, and $\Vmat{E}'$  \\
		$\gets$ \partialsvd($t_s$, $t_e$, $\Uset, \Sset, \Vset$) \label{alg:line:part}
		%\STATE  {$\mat{\Sigma}\matt{V} \gets$ vertical concatenation of $\mat{\Sigma'}_{(S)}\matt{V'}_{(S)}$, $\Smat{S+1}\Vmatt{S+1}$,  $\cdots$, and $\mat{\Sigma'}_{(E)}\matt{V'}_{(E)}$}
		\STATE $\mat{\Sigma}\matt{V} \gets [\mat{\Sigma'}_{(S)}\matt{V'}_{(S)}; \Smat{S+1}\Vmatt{S+1};\cdots;\mat{\Sigma'}_{(E)}\matt{V'}_{(E)}]$ \label{alg:opt_query_phase:vertical}
		\STATE $\mat{U}_{r}, \mat{\Sigma}_{r}, \matt{V}_{r} \gets$ low-rank approximation of $\mat{\Sigma}\matt{V}$ using SVD \label{alg:opt_query_phase:low_rank}
		\STATE $\Vmat{S:E} \gets\mat{V}_{r}$ and $\Smat{S:E} \gets \mat{\Sigma}_{r}$ \label{alg:opt_query_phase:set_v}
		%\STATE $\Smat{S:E} \gets \mat{\Sigma}_{r}$ \label{alg:opt_query_phase:set_sigma}
		\STATE $\Umat{S:E}$  $\gets$ $[
		\Umat{S}'\mat{U}'_{r(S)} ~;~ \Umat{S+1}\mat{U}_{r(S+1)} ~;~ \cdots ~;~ \Umat{E}'\mat{U}'_{r(E)}
		]$ \label{alg:opt_query_phase:block_diag}
		\RETURN $\Umat{S:E}$, $\Smat{S:E}$, and $\Vmat{S:E}$ \label{alg:opt_query_phase:return}
	\end{algorithmic}
\end{algorithm}

\subsubsection {\textbf{\partialsvd}}
  \label{subsec:partialSVD}

\vspace{-1.2mm}

This module manipulates the SVD results of block matrices $\Abmat{S}$ and $\Abmat{E}$ to return the SVD results in a given time range $[t_s, t_e]$.
As seen in Figure~\ref{fig:proposedmethod}, $\Abmat{S}$ may contain the time range before $t_s$, and $\Abmat{E}$ may include the time range after $t_e$.
Note that those time ranges are out of the time range of the given query;
thus, our goal for this module is to extract SVD results from $\Abmat{S}$ and $\Abmat{E}$ according to the time range query without reconstructing raw data.
Figure~\ref{fig:partial_svd} depicts the operation of \partialsvd.
For the block matrix $\Abmat{S}$ and its SVD $\Umat{S}\Smat{S}\Vmat{S}$,
\partialsvd first eliminates rows of left singular vector matrix $\Umat{S}$ which are out of the query time range.
After that, \partialsvd multiplies the remaining left singular vector matrix $\mat{X}_{s}\Umat{S}$ with the singular value matrix $\Smat{S}$, and performs SVD $\mat{\tilde{U}}_{(S)} \mat{\tilde{\Sigma}}_{(S)} \matt{\tilde{V}}_{(S)} \leftarrow \mat{X}_{s}\Umat{S}\Smat{S}$ of the resulting matrix.
The resulting singular vector matrix $\mat{\tilde{U}}_{(S)}$ and the singular value matrix $\mat{\tilde{\Sigma}}_{(S)}$ constitute the output of \partialsvd. 
The remaining right singular vector matrix output of \partialsvd is computed by multiplying the right singular vector matrix $\matt{\tilde{V}}_{(S)}$ with $\Vmatt{S}$.
Similar operations are performed for the block matrix $\Abmat{E}$ and its SVD $\Umat{E}\Smat{E}\Vmat{E}$.

Now, we describe the details of this module (Algorithm~\ref{alg:partial_svd}).
We first introduce elimination matrices which are used in \partialsvd to adjust the time range.	

\vspace{-1mm}

\begin{definition}[Elimination matrices]
Suppose $r_{S}$ is the number of rows to be eliminated in $\Abmat{S}$ according to $t_s$.
	Then $b_{S}=b-r_{S}$ is the number of remaining rows in $\Abmat{S}$.
	Similarly, let $r_{E}$ be the number of rows to be eliminated in $\Abmat{E}$ according to $t_e$; then $b_{E}=b-r_{E}$ is the number of remaining rows in $\Abmat{E}$.
The elimination matrices $\mat{X}_{s}$ and $\mat{X}_{e}$ for $\Abmat{S}$ and $\Abmat{E}$ are defined as follows:
	\begin{align}
		\label{eq:elimination_matrix}
		\mat{X}_{s} = \begin{bmatrix}
			\Zmat_{b_{S} \times r_{S}} & \Imat_{b_{S} \times b_{S}}
		\end{bmatrix}
			\mat{X}_{e} = \begin{bmatrix}
			\Imat_{b_{E} \times b_{E}} & \Zmat_{b_{E} \times r_{E}}
		\end{bmatrix}
	\end{align}
\QEDB
\end{definition}
\vspace{-3mm}

\begin{algorithm} [t]
	\caption{\partialsvd}\label{alg:partial_svd}
	\begin{algorithmic} [1]
		\algsetup{linenosize=\small}
		\small
%		\footnotesize
		\renewcommand{\algorithmicrequire}{\textbf{Input:}}
		\renewcommand{\algorithmicensure}{\textbf{Output:}}
		\REQUIRE starting point $t_{s}$, ending point $t_{e}$, and sets $\Uset$, $\Sset$, and $\Vset$ of SVD results
		\ENSURE SVD results $\Umat{S}'$, $\Smat{S}'$, $\Vmat{S}'$, $\Umat{E}'$, $\Smat{E}'$,  and $\Vmat{E}'$
 of $\Abmat{S}$ and $\Abmat{E}$ within the time range 		 %\renewcommand{\algorithmicrequire}{\textbf{Parameters:}}
		%\REQUIRE {the index of first block matrix: $S$, the index of last block matrix: $E$, the number of rows of matrix $\mathbf{X_s}$: $b_{S}$, the number of rows of $\mathbf{X_e}$: $b_{E}$}
		%\STATE \textbf{Initialize:} $b_{S} \gets t_s$, $b_{E} \gets t_e$
		\STATE $S \gets t_s/b$ and $E \gets t_e/b$ \label{alg:partial_svd:SE}
		\STATE $b_S \gets b-t_s\%b$ and $b_E \gets t_e\%b$ (\%: the modulus operator) \label{alg:partial_svd:bSE}
		\STATE construct $\mat{X}_{s}$ and $\mat{X}_{e}$ based on $b_{S}$ and $b_{E}$ as in Equation~\eqref{eq:elimination_matrix} \label{alg:partial_svd:elimination_matrice}
		\STATE $\mat{\tilde{U}}_{(S)}$, $\mat{\tilde{\Sigma}}_{(S)}$, $\matt{\tilde{V}}_{(S)}$ $\gets$ low-rank approximation of $\mat{X}_{s}\Umat{S}\Smat{S}$ using SVD \label{alg:partial_svd:low_rank:starting}
		\STATE $\Umat{S}' \gets \mat{\tilde{U}}_{(S)}$, $\Smat{S}' \gets \mat{\tilde{\Sigma}}_{(S)}$, and  $\Vmattt{S} \gets \matt{\tilde{V}}_{(S)} \Vmatt{S}$ \label{alg:partial_svd:set:starting}
		\STATE $\mat{\tilde{U}}_{(E)}$, $\mat{\tilde{\Sigma}}_{(E)}$, $\matt{\tilde{V}}_{(E)}$ $\gets$ low-rank approximation of $\mat{X}_{e}\Umat{E}\Smat{E}$ using SVD \label{alg:partial_svd:low_rank:ending}
		\STATE $\Umat{E}' \gets \mat{\tilde{U}}_{(E)}$, $\Smat{E}' \gets \mat{\tilde{\Sigma}}_{(E)}$, and $\Vmattt{E} \gets \matt{\tilde{V}}_{(E)} \Vmatt{E}$  \label{alg:partial_svd:set:ending}
		\RETURN $\Umat{S}'$, $\Smat{S}'$, $\Vmat{S}'$, $\Umat{E}'$, $\Smat{E}'$, and $\Vmat{E}'$
	\end{algorithmic}
\end{algorithm}

The matrices $\Abmat{S}$ and $\Abmat{E}$ are multiplied to the elimination matrices,
and the time ranges of the resulting matrices $\mat{X}_{s}\Abmat{S}$ and $\mat{X}_{e}\Abmat{E}$ are within the query time range $[t_s, t_e]$.
\partialsvd constructs those elimination matrices based on $b_{S}$ and $b_{E}$ (line~\ref{alg:partial_svd:elimination_matrice} of Algorithm~\ref{alg:partial_svd}).
The filtered block matrix $\mat{X}_{s}\Abmat{S}$ is given by
\vspace{-1.5mm}
\begin{equation}
	\label{eq:partial:step_1}
	\mat{X}_{s}\Abmat{S} \simeq \mat{X}_{s}(\Umat{S}\Smat{S}\Vmatt{S}) = (\mat{X}_{s}\Umat{S}\Smat{S})\Vmatt{S}
\end{equation}
\vspace{-2mm}
\noindent where $\Abmat{S} \simeq  \Umat{S}\Smat{S}\Vmatt{S}$ was computed at the storage phase.

\vspace{0.8mm}

\noindent \partialsvd decomposes $\mat{X}_{s}\Umat{S}\Smat{S}$ into $\mat{\tilde{U}}_{(S)}\mat{\tilde{\Sigma}}_{(S)}\matt{\tilde{V}}_{(S)}$ via SVD and low-rank approximation with threshold $\xi$ {since $\mat{X}_{s}\Umat{S}$ is not a column orthogonal matrix, and $\mat{X}_{s}\Umat{S}\Smat{S}\Vmatt{S}$ is not a form of the SVD result}; then, Equation~\eqref{eq:partial:step_1} is written as follows:
\begin{equation}
	\label{eq:partial:step_2}
\small\hspace{-3mm}(\mat{X}_{s}\Umat{S}\Smat{S})\Vmatt{S} \! \simeq \!\mat{\tilde{U}}_{(S)}\mat{\tilde{\Sigma}}_{(S)}(\mat{\tilde{V}}_{(S)}\Vmat{S})^\text{T} \!=\! \Umat{S}'\Smat{S}'\Vmattt{S}
\end{equation}
\noindent where $\Umat{S}'=\mat{\tilde{U}}_{(S)}$, $\Smat{S}'=\mat{\tilde{\Sigma}}_{(S)}$, and $\Vmattt{S}=\matt{\tilde{V}}_{(S)}\Vmatt{S}$.
%Note that $\Vmat{S}$ is an orthogonal matrix. % by Lemma~\ref{lemma:Columnorthogonal}.
In line~\ref{alg:partial_svd:low_rank:starting}, \partialsvd performs SVD on $\mat{X}_{s}\Umat{S}\Smat{S}$, and in line~\ref{alg:partial_svd:set:starting}
it computes $\Vmattt{S}=\matt{\tilde{V}}_{(S)}\Vmatt{S}$. % (line~\ref{alg:partial_svd:set:starting}).
\partialsvd similarly computes the SVD result of $\mat{X}_{e}\Abmat{E}$ in lines \ref{alg:partial_svd:low_rank:ending}$\sim$\ref{alg:partial_svd:set:ending} of Algorithm~\ref{alg:partial_svd}.

\vspace{-1mm}

\subsubsection{\textbf{\stitchedsvd}}
\label{subsec:block}
%According to Lemma~\ref{lemma:time_complexity:bais_stitched_svd}, the query phase (Algorithm~\ref{alg:queryphase}) exploiting basic \stitchedsvd is not efficient for a long time range query since the time complexity is mainly proportional to $l^2$.
This module combines the \partialsvd of $\Abmat{S}$ and $\Abmat{E}$, and the stored SVD results of blocks matrices $\Abmat{S+1}, \Abmat{S+1}, \cdots, \Abmat{E-1}$ in the query time range $[t_{s}, t_{e}]$ to return the final SVD result corresponding to the query range as shown in Figure~\ref{fig:proposedmethod}.
{A naive approach is to reconstruct the data blocks using the stored SVD results and perform SVD on the reconstructed data of the given query time range.
However, this approach cannot provide fast query speed for a long time range due to heavy computations induced by the reconstruction and the following SVD.
The goal of \stitchedsvd is to efficiently stitch the SVD results in the query time range by avoiding reconstruction and minimizing the numerical computation of matrix multiplication.}
%The goal of \stitchedsvd is to efficiently stitch the stored SVD results in the query time range.

Specifically, \stitchedsvd stitches several consecutive block SVD results together to compute the SVD corresponding to the query time range: is.e., it combines the SVD result $\Umat{i}$, $\Smat{i}$, and $\Vmat{i}$ of the $i$th block matrix $\Abmat{i}$, for $i=S..E$, to compute the SVD $\Umat{S:E}$, $\Smat{S:E}$, and $\Vmat{S:E}$.
The main idea is 1) to carefully decouple the matrices $\Umat{i}$ from $\Smat{i}$$\Vmat{i}$,
2) construct a stacked matrix containing $\Smat{i}$$\Vmat{i}$ for $i=S..E$,
3) perform SVD on the stacked matrix to get the singular value matrix and the right singular vector matrix of the final SVD result,
and 4) carefully combine $\Umat{i}$ with the left singular matrix of SVD of the stacked matrix to get the left singular vector matrix of the final SVD result.

%The formal definition of stitched SVD matrices is as follows.
%\begin{definition}[Stitched SVD matrices]
%	Let $\Umat{i}$, $\Smat{i}$, and $\Vmat{i}$ be the SVD result of $i$-th block matrix $\Abmat{i}$.
%Also, let
%	$\Umat{i:j}$ be the left singular vector matrix stitched from $\Umat{i}$ to $\Umat{j}$,
%$\Smat{i:j}$ be the singular value matrix stitched from $\Smat{i}$ to $\Smat{j}$,
%and $\Vmat{i:j}$ be the right singular vector matrix from $\Vmat{i}$ to $\Vmat{j}$.
%	We call $\Umat{i:j}$, $\Smat{i:j}$, and $\Vmat{i:j}$ as the stitched SVD matrices from $\Abmat{i}$ to $\Abmat{j}$.
%\end{definition}

Lines 2 to 5 of Algorithm~\ref{alg:opt_query_phase} present how stitched SVD matrices are computed.
%Suppose $\Abmat{S:i}$ is the block matrix concatenated from $S$-th block to $i$-th block, and \stitchedsvd is given the stitched SVD matrices from block index $S$ to $i$ (i.e., $\Umat{S:i}$, $\Smat{S:i}$, and $\Vmat{S:i}$ for $\Abmat{S:i}$) and the SVD result of $(i+1)$-th block matrix (i.e., $\Umat{i+1}$, $\Smat{i+1}$, and $\Vmat{i+1}$) as inputs.
%Then, $\Abmat{S:i+1}$ satisfies the following equations.
%we optimize the query phase by improving the performance of \stitchedsvd in terms of time (Algorithm~\ref{alg:opt_query_phase}).
%The main problem of basic \stitchedsvd is the part for computing the left singular vector matrix $\Umat{S:i+1}$ (line~\ref{alg:stitched_svd:basic:U} in Algorithm~\ref{alg:stitched_svd:basic}) since the number of rows of $\Umat{S:i}$ increases as iteration $i$ proceeds; thus, the time for the matrix multiplication increases.
%To avoid this problem,
%we construct block matrix structure which splits a matrix into several small blocks (i.e., $\footnotesize \mathbf{A} = {\begin{bmatrix}
%\mathbf{A^1} ~;~ \mathbf{A^2} ~;~ \cdots ~;~ \mathbf{A^N} \end{bmatrix}}^T$).
%The block matrix structure has improved the performance of matrix operations on parallel or distributed environments~\cite{DBLP:journals/siammax/IwenO16}.
First, we construct $\footnotesize \mat{\Sigma}\matt{V}$ based on the block matrix structure where $\footnotesize \mat{\Sigma}\matt{V}$ is equal to $\footnotesize {\begin{bmatrix} \mat{\Sigma'}_{(S)}\matt{V'}_{(S)} ~;~ \Smat{S+1}\Vmatt{S+1} ~;~  \cdots ~;~ \mat{\Sigma'}_{(E)}\matt{V'}_{(E)} \end{bmatrix}}^T$.
%Note that the block matrix structure which splits a matrix into several small blocks (i.e., $\footnotesize \mathbf{A} = {\begin{bmatrix}
%\mathbf{A^1} ~;~ \mathbf{A^2} ~;~ \cdots ~;~ \mathbf{A^N} \end{bmatrix}}^T$) has improved the performance of matrix operations on parallel or distributed environments~\cite{DBLP:journals/siammax/IwenO16}.
After organizing the block matrix structure of $\footnotesize \mat{\Sigma}\matt{V}$, we define block diagonal matrix $diag(S:E)$ as follows.
\begin{definition}[Block diagonal matrix]
	\label{def:blockdiagonal}
	Suppose $\Umat{S}'$ and $\Umat{E}'$ are the left singular vector matrices produced by \partialsvd.
%	 where $S$ is the starting block index, and $E$ is the ending block index for a given time range query $[t_s, t_e]$.
Let $\Umat{S+1}$, $\Umat{S+2}$, $\cdots$, $\Umat{E-1}$ be the left singular vector matrices in $\Uset$.
{The block diagonal matrix $diag(S:E)$ is defined as follows:}
	%\vspace{-5mm}
	\begin{equation}
	\label{blockdiag}
	\footnotesize
%	\resizebox{0.8\linewidth}{!} {
		diag(S:E) = \begin{bmatrix}
		\Umat{S}' & \mathbf{O} & \cdots & \Zmat \\
		\Zmat &\Umat{S+1} &  & \vdots \\
		\vdots  &  & \ddots & \vdots   \\
%		\vdots &  &  &\Umat{E-1} & \Zmat  \\
		\Zmat & \cdots & \cdots & \Umat{E}' \\
		\end{bmatrix} \nonumber
	\end{equation}
\QEDB
	\vspace{-2mm}
\end{definition}

{Then, the matrix corresponding to the time range query $[t_s, t_e]$ is represented as follows:}
\begin{align}
\footnotesize
	\label{eq:stitched_svd:optimized:block}
	\begin{split}
		\Abmat{S:E} &= \begin{bmatrix}
			\mat{X}_{s}\Abmat{S} \\
			\Abmat{S+1} \\
			\cdots\\
			\mat{X}_{e}\Abmat{E}
		\end{bmatrix}
		\simeq
		\begin{bmatrix}
			\Umat{S}'\Smat{S}'\Vmattt{S}\\
			\Umat{S+1}\Smat{S+1}\Vmatt{S+1}\\
			\cdots \\
			\Umat{E}'\Smat{E}'\Vmattt{E}\\
		\end{bmatrix}
		= diag(S:E)\mat{\Sigma}\matt{V}
	\end{split}	
\end{align}
{\noindent where $\mat{X}_{s}$ and $\mat{X}_{e}$ are elimination matrices of \partialsvd, and $\mat{\Sigma}\matt{V}$ is equal to $\footnotesize {\begin{bmatrix} \mat{\Sigma'}_{(S)}\matt{V'}_{(S)} ~;~ \Smat{S+1}\Vmatt{S+1} ~;~  \cdots ~;~ \mat{\Sigma'}_{(E)}\matt{V'}_{(E)} \end{bmatrix}}^T$.
%Note that the result of Equation~\eqref{eq:stitched_svd:optimized:block} is similar to that of Equation~\eqref{eq:stitched_svd:basic:block}.
As we apply SVD and low-rank approximation to $\mat{\Sigma}\matt{V}$, Equation~\eqref{eq:stitched_svd:optimized:block} becomes as follows:}

\vspace{-4mm}
\begin{align}
\label{eqn:diag_eqn}
	\begin{split}
		diag(S:E)\mat{\Sigma}\matt{V} &\simeq diag(S:E)\mat{U}_{r}\mat{\Sigma}_{r}\matt{V}_{r} \\
		&= \Umat{S:E}\Smat{S:E}\Vmatt{S:E}
	\end{split}
\end{align}
\noindent where {$\mat{\Sigma}\matt{V}\simeq\mat{U}_{r}\mat{\Sigma}_{r}\matt{V}_{r}$ is computed by low-rank approximation and SVD, $\Umat{S:E}=diag(S:E)\mat{U}_{r}$, $\Smat{S:E}=\mat{\Sigma}_{r}$, and $\Vmat{S:E}=\mat{V}_{r}$.}
%{ $\Umat{S:E}\Smat{S:E}\Vmatt{S:E}$ is equivalent to the direct SVD result of $\Abmat{S:E}$ since it shows the same property of the SVD as shown in~\cite{Supplementary}.}
To avoid matrix multiplication between $\mat{U}_{r}$ and zero sub-matrices of $diag(S:E)$,
	we split $\mat{U}_{r}$ block by block as follows:
{
\begin{equation*}
\footnotesize
\mat{U}_r = {\begin{bmatrix} \mat{U}'_{r(S)} ~;~ \mat{U}_{r(S+1)} ~;~ \cdots ~;~ \mat{U}'_{r(E)}\end{bmatrix}}^T
\end{equation*}
\noindent where $\mat{U}'_{r(S)}$ and $\mat{U}'_{r(E)}$ correspond to $\Umat{S}'$ and $\Umat{E}'$, respectively, and $\mat{U}_{r(i)}$ correspond to $\Umat{i}$ for $S+1 \leq i \leq E-1$.
Then $\Umat{S:E}=diag(S:E)\mat{U}_{r}$ of Equation~\eqref{eqn:diag_eqn} is computed as follows:
\begin{align}
	\label{blockdiag2}
	\begin{split}
	&\Umat{S:E}=diag(S:E)\mat{U}_{r}\\ &=
	%\resizebox{0.85\linewidth}{!} {
	{\begin{bmatrix}
		\Umat{S}'\mat{U}'_{r(S)} ~;~ \Umat{S+1}\mat{U}_{r(S+1)} ~;~ \cdots ~;~ \Umat{E}'\mat{U}'_{r(E)}
		\end{bmatrix}}^T
	%}
	\end{split}
\end{align}	
}

\vspace{-2mm}

The column orthogonality of $\Umat{S:E}=diag(S:E)\mat{U}_{r}$ is established as it is the product of two column orthogonal matrices; also, $\Vmat{S:E}=\mat{V}_{r}$ is column orthogonal.
{Note that we perform \partialsvd to satisfy column orthogonal condition before performing \stitchedsvd.}
\subsection{Theoretical Analysis}
%\vspace{-1mm}
\label{subsec:Theo}
We theoretically analyze our proposed method \method in terms of time and memory cost.
%Table \ref{tab:Theoretical} shows the summary of time and space complexity.
Note that a collection of multiple time series data $\mathbf{A}$ is a dense matrix, and the time complexity to compute SVD of $\mat{A} \in \mathbb{R}^{t \times c}$ is $O(min(t^2c, tc^2))$.
%%\begin{table} [h]
%%	\centering
%%	\caption{
%%		Summary of theoretical Analysis.
%%		$c$ is the column dimension of collection of multiple stream data, $k$ is the rank computed by SVD, $t$ is the total time length, $x$ is equal to $(c^2/b+k)$, and $t_s$ and $t_e$ are the start and end points of time range query.}
%%	\label{tab:Theoretical}
%%	\begin{tabular}{c|c|c|c|c}
%%		\toprule
%%		{ } & \multicolumn{2}{c|}{\textbf{naive SVD}} & \multicolumn{2}{c}{\textbf{\method}} \\
%%		{ } & \textbf{Store} & \textbf{Query} & \textbf{Store} & \textbf{Query} \\
%%		\midrule
%%		\textbf{Time}  & O($ckt$) &  O($c^2(t_e-t_s)$) & O($k^2(c+b)$)  & {O($kx$$(t_e-t_s)$)}\\ \hline
%%		\textbf{Space} & O($ct$) &  O($c(t_e-t_s)$) & O($kb$) & O($k(t_e-t_s)$)\\
%%		\bottomrule
%%	\end{tabular}
%%\end{table}
%

\textbf{Time Complexity}.
We analyze the time complexities of the storage and the query phases in Theorems \ref{theo:time of store} and \ref{theo:time of Query}, respectively.
{
\begin{theorem}
	\label{theo:time of store}
When a vector $\vect{a}_{t} \in \mathbb{R}^{1 \times c}$ is given at time $t$, the computation cost of storage phase in \method is $O(k^2(b+c))$, where $k$ is the number of singular values.
%where $k_{t-1}$ is the rank of SVD result at time $t-1$, and $k_{t}$ is the rank of SVD results at time $t$.
%O($min(c^{2}(1+k_{t-1}),c(1+k_{t-1})^{2})$+ $b (k_{t-1}+1) k_{t}$)
\end{theorem}
\begin{proof}
	Performing SVD of {\footnotesize${\begin{bmatrix} \Smat{i,t-1}\Vmatt{i,t-1} ~;~ \vect{a}_{t} \end{bmatrix}}$} takes $O(min(c^{2}$
	$(1+k_{t-1}),c(1+k_{t-1})^{2})$, and multiplication of {\footnotesize$\begin{bmatrix}
	\Umat{i,t-1} & \Zmat \\
	\Zmat & \Imat	
	\end{bmatrix}$} and $\mat{\tilde{U}}$ takes $O(b(k_{t-1}+1)k_{t})$ since the row length of {\footnotesize$\begin{bmatrix}
	\Umat{i,t-1} & \Zmat \\
	\Zmat & \Imat	
	\end{bmatrix}$} is always smaller than or equal to the block size $b$.
	Assume $k_{t-1}$ and $k_{t}$ are equal to $k$. %, and $k+1 \simeq k$.
	The total computational cost of storage phase in \method is O($min(c^{2}k,ck^{2})$+ $bk^2$).
	We simply express the computational cost of storing the incoming data at each time tick as $O(k^2(b+c))$ since the number $c$ of columns is generally greater than $k$.
\end{proof}
In Theorem~\ref{theo:time of store}, the computation of storing the incoming data at each time tick takes constant time since $b$ and $c$ are constants and $k$ is smaller than $c$.
}
{
\begin{theorem}
	\label{theo:time of Query}
	Given a time range query $[t_s, t_e]$, the time cost of query phase (Algorithm \ref{alg:opt_query_phase}) is $O((t_e-t_s)k(k+\frac{c^2}{b})+{k^2}(b+c))$.
%O ( & min\left(c^2\left(k'_{(S)}+k'_{(E)}+{\sum_{i=S+1}^{E-1} k_i}\right), c{\left(k'_{(S)}+k'_{(E)}+{\sum_{i=S+1}^{E-1} k_i}\right)}^2\right) \\
%& + \left(bk_{(S:E)}\left(k'_{(S)}+k'_{(E)}+{\sum_{i=S+1}^{E-1} k_{(i)}} \right)+ (b-r_{S}+c)k'^2_{(S)}+ (b-r_{E}+c)k'^2_{(E)}\right))

\end{theorem}

\begin{proof}
It takes $O((b-r_{S}+c)k'^2_{(S)} + (b-r_{E}+c)k'^2_{(E)})$ to compute \partialsvd where $k'_{(S)}$ and $k'_{(E)}$ are the number of singular values computed by \partialsvd (line~\ref{alg:line:part} in Algorithm~\ref{alg:opt_query_phase}).
The computational time to perform SVD of $\footnotesize {\begin{bmatrix} \mat{\Sigma'}_{(S)}\matt{V'}_{(S)} ~;~ \Smat{S+1}\Vmatt{S+1} ~;~ \cdots ~;~ \mat{\Sigma'}_{(E)}\matt{V'}_{(E)} \end{bmatrix}}$ depends on $O(min(c^2(k'_{(S)}+k'_{(E)}+{\sum_{i=S+1}^{E-1} k_{(i)}})$, $c(k'_{(S)}+k'_{(E)}+{\sum_{i=S+1}^{E-1} k_{(i)}})^2)$ in \stitchedsvd since horizontal and vertical length of the matrix are $c$ and $(k'_{(S)}+k'_{(E)}+{\sum_{i=S+1}^{E-1} k_{(i)}})$, respectively (lines~\ref{alg:opt_query_phase:vertical}$\sim$~\ref{alg:opt_query_phase:set_v} in Algorithm~\ref{alg:opt_query_phase}).
Also, the computational time of block matrix multiplication for $\Umat{S:E}$ (line~\ref{alg:opt_query_phase:block_diag} in Algorithm~\ref{alg:opt_query_phase}) takes O($k_{(S:E)}$\\($(b-r_{S})k'_{(S)}$$+(b-r_{E})k'_{(E)}+b{\sum_{i=S+1}^{E-1}k_{(i)}}$)) where $k_{(S:E)}$ is the number of singular values with respect to the SVD result of a given time range $[t_s,t_e]$.
%The query cost is proportional to the sum of computation time of \stitchedsvd, and \partialsvd.
Let all $k_{(i)}$'s be $k$ in query phase, $k'_{(S)}+k'_{(E)}+{\sum_{i=S+1}^{E-1} k_{(i)}}$ be larger than $c$; also, replace $b-r_{i}$ with block size $b$ since $b$ is always greater than $b-r_{i}$.
Then, the computational time of \partialsvd and \stitchedsvd takes $O(k^2(b+c))$ and $O((t_e-t_s)k(k+\frac{c^2}{b}))$, respectively.
We can simply express the computational cost of \method as $O((t_e-t_s)k(k+\frac{c^2}{b})+{k^2}(b+c))$.
\end{proof}

%The query cost is proportional to the sum of computation time of \stitchedsvd, and \partialsvd.
%Let $k_{max}$ be the maximum value of all $k_{(i)}$'s in query phase, and $k'_{(S)}+k'_{(E)}+{\sum_{i=S+1}^{E-1} k_{(i)}}$ be larger than $c$.
Theorem~\ref{theo:time of Query} implies that the computational time of \method in query phase linearly depends on the time range $(t_e-t_s)$.
}

\textbf{Space Complexity.} We analyze the space complexity of the storage phase in Theorem~\ref{theo:space of store}.
{
\begin{theorem}
	\label{theo:space of store}
{	Space complexity of \method for storing data is $O(tk(1+\frac{k}{b}+\frac{c}{b}))$ where $t$ is the total time length, and $k$ is the number of singular values.}
%	Space complexity of \method for storing data takes $O(\sum_{i=1}^{\lfloor \frac{t}{b} \rfloor}k_{(i)}(b+k_{(i)}+c))$ when total time length is $t$.
\end{theorem}
\begin{proof}
	At time $t$, we have $\lfloor \frac{t}{b} \rfloor$ SVD results where $\Umat{i}\in\mathbb{R}^{b_i \times k_{(i)}}$, $\Smat{i}\in\mathbb{R}^{k_{(i)} \times k_{(i)}}$, and $\Vmat{i}\in\mathbb{R}^{k_{(i)} \times c}$.
	Therefore, the number of elements in every matrix we have is $O(\sum_{i=1}^{\lfloor \frac{t}{b} \rfloor}k_{(i)}(b+k_{(i)}+c))$. Assuming $k_{(i)}$ is equal to $k$, we briefly express the space cost of \method as $O(tk(1+\frac{k}{b}+\frac{c}{b}))$.
\end{proof}
%Let $s_{max}$ be the maximum value of all $k_{(i)}$'s; assuming $b > c$, we briefly express the space cost of \method as $O(s_{max}t)$ in Theorem~\ref{theo:space of store}.
The space cost linearly depends on $t$ and $k$, but $k$ is  much smaller than the number $c$ of time series.
Therefore, \method efficiently compresses incoming data using space linear to time length with low coefficient.
}
%\subsection{TR-SVD-M:Storing multiple time series data with multiple block size.}
%\label{sec:multipleblock}
%The first challenge for efficient TR-SVD-B is how to minimize the amount of multiple time series data to store in storage phase.
%In Section \ref{subsec:Collection}, we store the SVD result of matrices with block size $b$.
%The number of block sized SVD results affects the space efficiency because the number of singular matrix $\mathbf{\Sigma}$ and right singular matrix $\mathbf{V}$ is inversely proportional to block size $b$.
%The main idea is that we store the SVD result of block matrix with multiple size instead of fixed-size block.
%Assume that the data we recently received are more interesting than the old one, there is almost no variation in the number of singular values $k$, and either one or two SVD results with the same power of 2 number of block size.
%we repeatedly perform those steps to tackle the challenge.
% \begin{enumerate}
% 	\item we update the SVD result in incremental SVD until getting the SVD result of a matrix with block size $b$,, as soon as the number of multiple time series data has arrived.
% 	\item we combine the oldest two into a block matrix with block size $2b$ using Integrated SVD if there are three SVD results with the same block size $b$.
% 	\item we repeat the second step for the larger block size in ascending order.
% \end{enumerate}
%Consequently, this procedure reduce the number of SVD results, which affects the space efficiency.

\section{Experiment}
\label{sec:experiment}

\begin{table}[t]
	
	\caption{Description of real-world multiple time series datasets. Each dataset is a matrix in $\mathbb{R}^{t \times c}$ where $t$ corresponds to total length (time), and $c$ corresponds to the number of time series.
	}
	\centering
	\label{tab:Description}
	\begin{tabular}{lrr}
		
		\toprule
		\textbf{Dataset} & \textbf{Total length ($t$)} & \textbf{Attribute ($c$)} \\
		\midrule
		 Activity\tablefootnote{\url{https://archive.ics.uci.edu/ml/datasets/PAMAP2+Physical+Activity+Monitoring}}  & 382,000  & 41\\
		 Gas\tablefootnote{\url{https://archive.ics.uci.edu/ml/datasets/Gas+sensor+array+under+dynamic+gas+mixtures}} &  4,107,000 & 16 \\
		 London\tablefootnote{\url{http://www.londonair.org.uk/london/asp/publicdetails.asp}} & 350,000 & 10 \\
%		Synthetic & $10^8$ & $16$ \\
		\bottomrule
	\end{tabular}
\end{table}
%\vspace{-1mm}
%\begin{figure*} [t]
%	\subfloat[Activity monitoring]{\includegraphics[width=0.33\textwidth]{FIG/FINAL/QUERY/ACTIVITY_OPT}}
%	\subfloat[Gas sensor]{\includegraphics[width=0.33\textwidth]{FIG/FINAL/QUERY/HTSENSOR_OPT}}
%	\subfloat[London air quality]{\includegraphics[ width=0.33\textwidth]{FIG/FINAL/QUERY/LONDON_OPT}} \\
%	\caption{We compare query times of optimized TR-SVD and basic TR-SVD. (a), (b), and (c) show the query times of the three datasets.
%		The starting and the ending points of the query are arbitrarily chosen, and we increase the length of the query range from $10^4$ to $2 \times 10^5$ by step size $5 \times 10^3$.
%		The query time of optimized TR-SVD is up to $39 \times$ faster than basic TR-SVD on Activity dataset. For the Gas sensor and London datasets, optimized TR-SVD is up to $63 \times$ and $54 \times$ faster than basic TR-SVD, respectively.
%	}
%	\label{fig:optimization}
%\end{figure*}
%We evaluate our proposed method experimentally. We answer the following questions:

We aim to answer the following questions to evaluate the performance of our method \method from experiments.

\begin{itemize}
  \item \textbf{Q1. Time cost (Section \ref{subsec:time_cost}).} How quickly does \method process time range queries compared to other methods?
  \item \textbf{Q2. Space cost (Section \ref{subsec:space_cost}).} How much memory space does \method require for compressed blocks?
 		%How efficient is the memory usage of \method compared to other methods?
%  \item \textbf{Q3. Accuracy (Section \ref{subsec:Accuracy}).} How accurate is the query phase of \method compared to other methods?
  \item \textbf{Q3. Time-Space-Accuracy Trade-off (Section \ref{subsec:tradeoff}).} What are the tradeoffs between query time, space, and accuracy by \method compared to other baselines?
  \item \textbf{Q4. Parameter (Section \ref{subsec:param_sense}).} How does the block size $b$ in Algorithm \ref{alg:storage_phase} affect the performance of \method in terms of time and space?
\end{itemize}

\subsection{Experiment Settings}
%We describe the experiment settings for \method.
\label{subsec:setting}

\begin{figure} [t]
	\subfloat[Running time of \method's storage phase in Activity dataset]{\includegraphics[ width=0.22\textwidth]{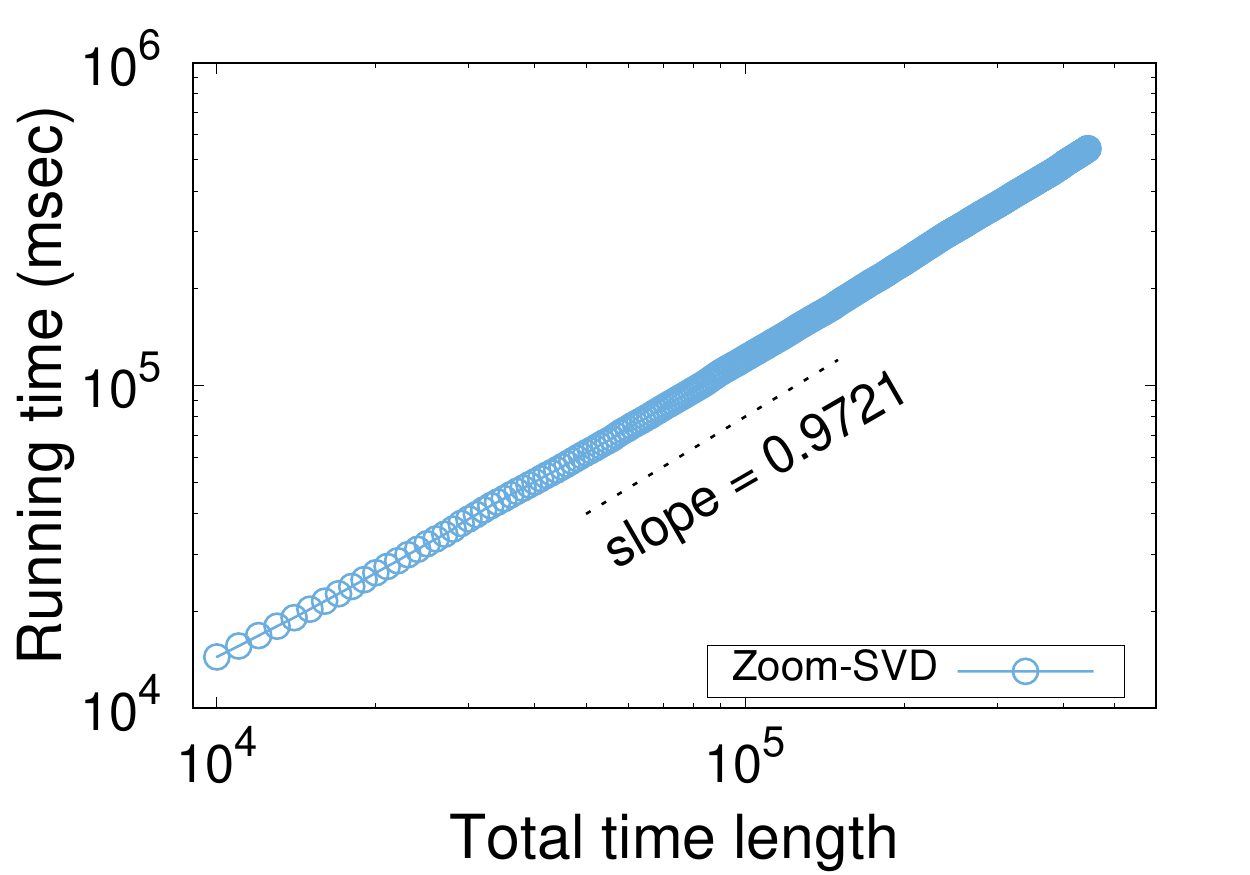}\label{fig:Store_Activity}}
	\subfloat[Running time of \method's storage phase in Gas dataset]{\includegraphics[ width=0.22\textwidth]{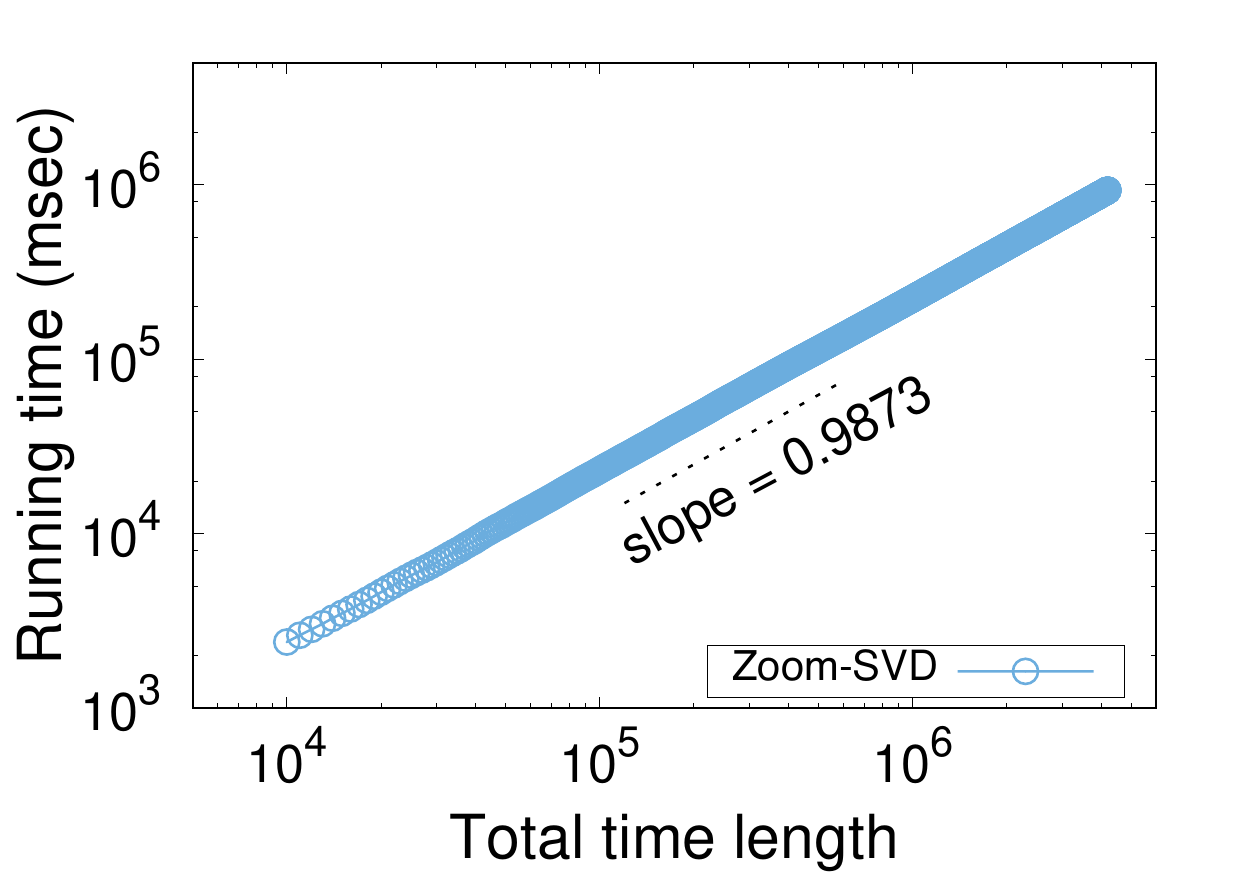}\label{fig:Store_HT}} \\
	%	\subfloat[Running time of \method's storing phase in London]{\includegraphics[ width=0.33\textwidth]{FIG/FINAL/STORE/STORING_LONDON}\label{fig:Store_London}}
	
	\caption{
		Running time of \method's storage phase on Activity and Gas datasets. The results show that the time to update all data is linear to the total time length $t$.
		The reason is that \method incrementally reads a new input vector, and computes the SVD of the block matrix in a constant time.
		The pattern on London dataset is similar to the above results.
		%		(b) and (c) have similar storage time, and degree of (a) have higher than that of (b) and (c) since column length of Gas sensor and London air quality dataset is almost the same and lower than that of Activity monitoring dataset.
	}
	\label{fig:StoringTime}
\end{figure}

\begin{figure}[t]
	\centering
	\includegraphics[width=0.28\textwidth]{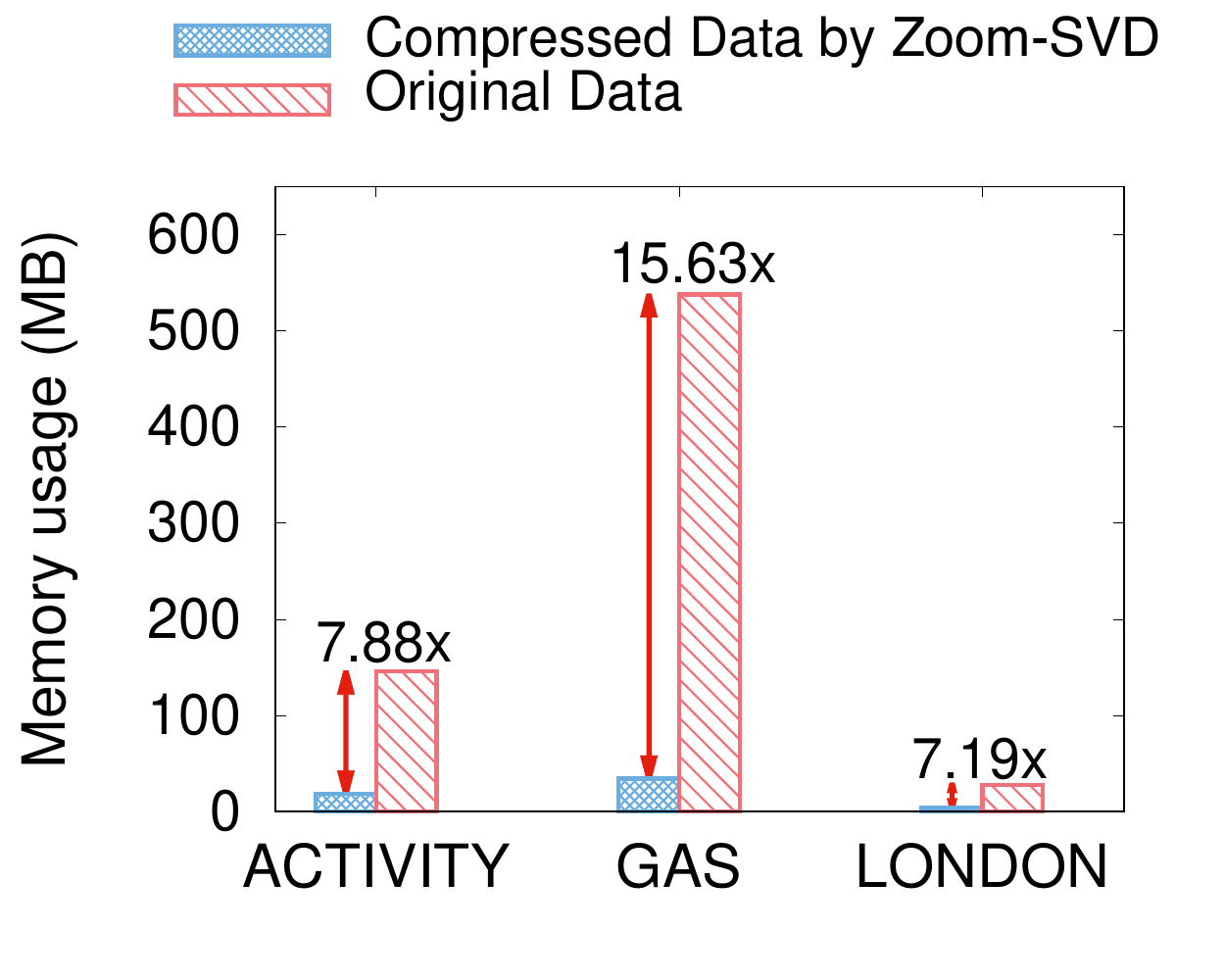}\label{fig:Compression}
	\caption
{Space savings by \method.
\method requires $7.88 \times$, $15.63 \times$, and $7.19 \times$ less space than the original data require for Activity, Gas sensor, and London, respectively.
}
	\label{fig:space_cost}
\end{figure}

\begin{figure*} [t]
	\vspace{-4mm}
	\subfloat{\includegraphics[width=0.8\textwidth]{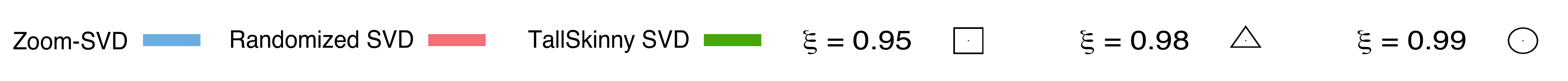}}\\
	\vspace{-4mm}
	\setcounter{subfigure}{0}
	\subfloat[Query time vs Error in Activity]{\includegraphics[ width=0.25\textwidth]{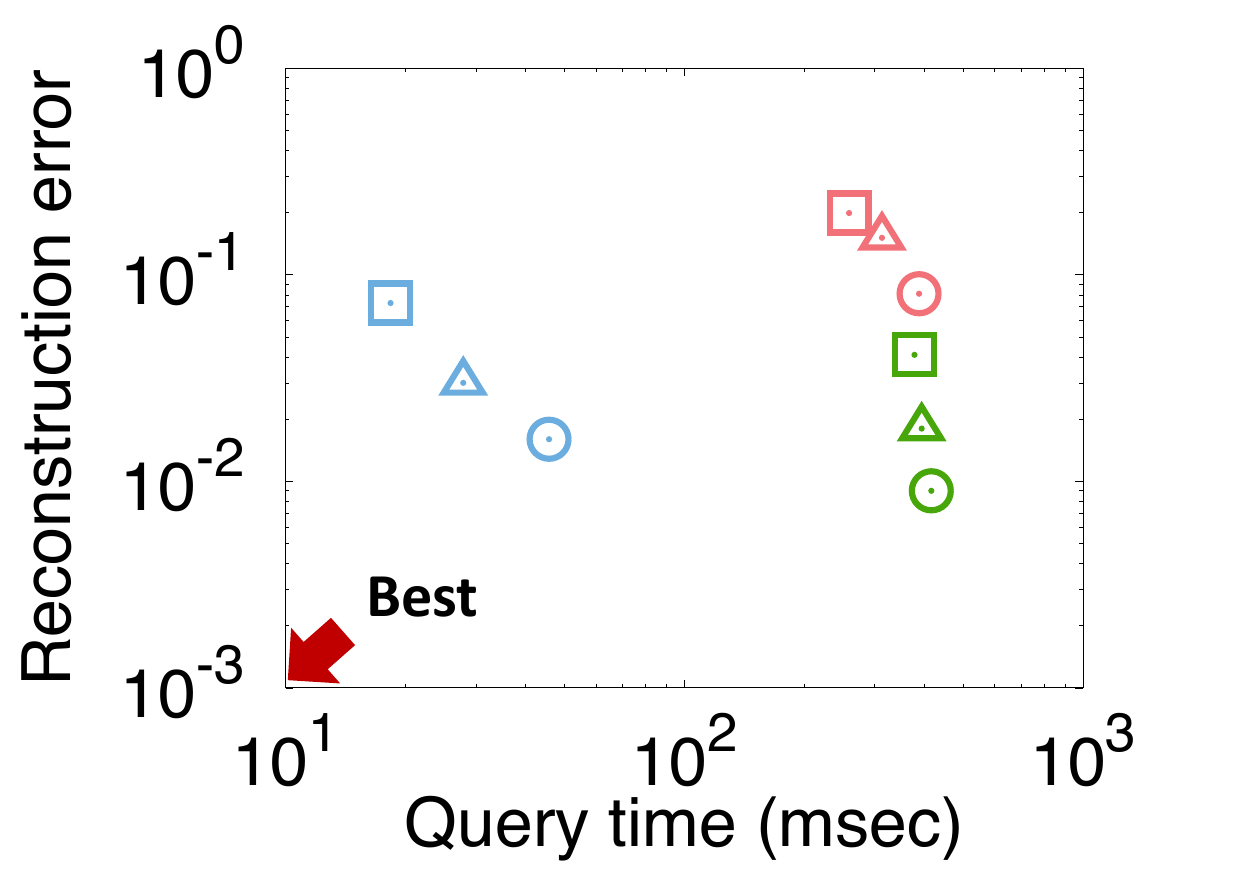}\label{fig:TRADEOFF_AT1}}
	\subfloat[Query time vs Error in Gas]{\includegraphics[ width=0.25\textwidth]{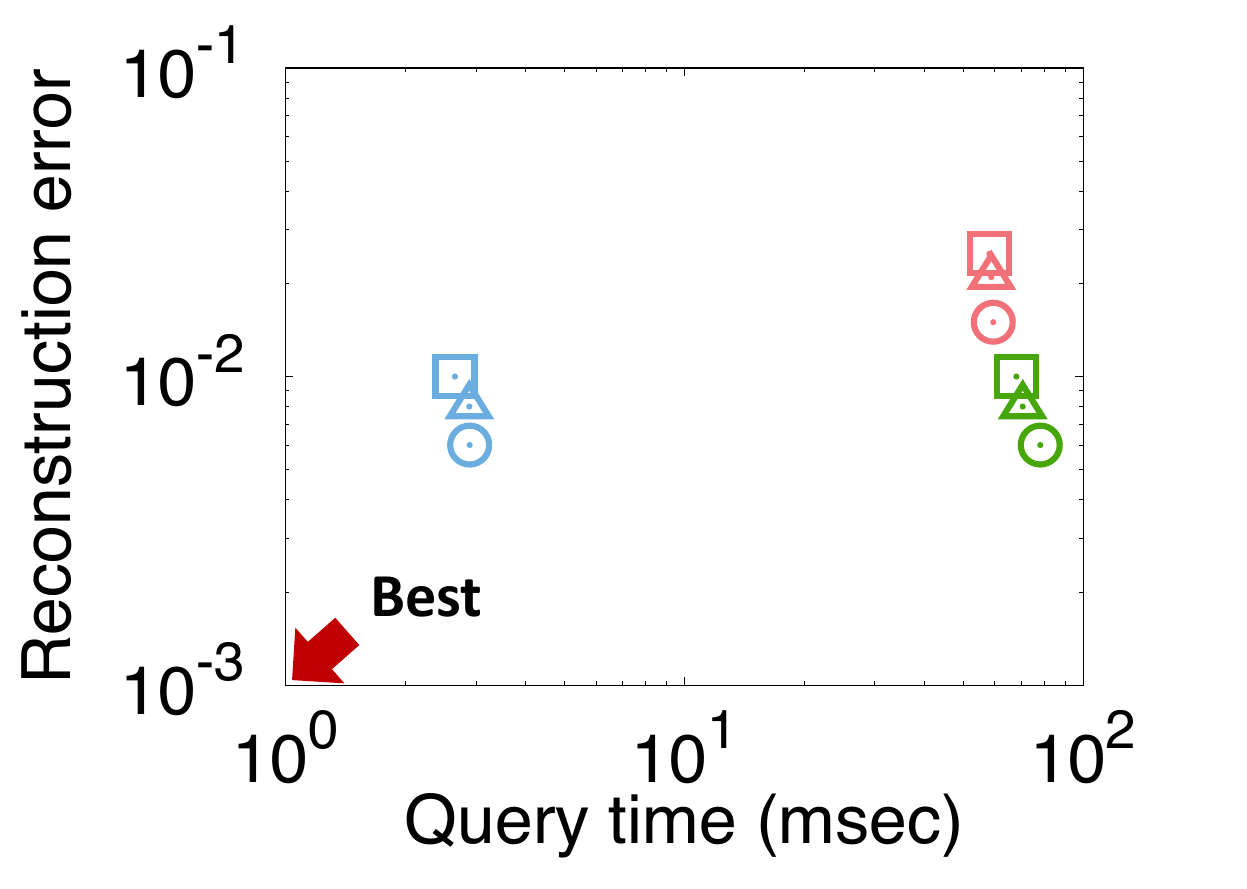}\label{fig:TRADEOFF_AT2}}
	\subfloat[Query time vs Error in London]{\includegraphics[ width=0.25\textwidth]{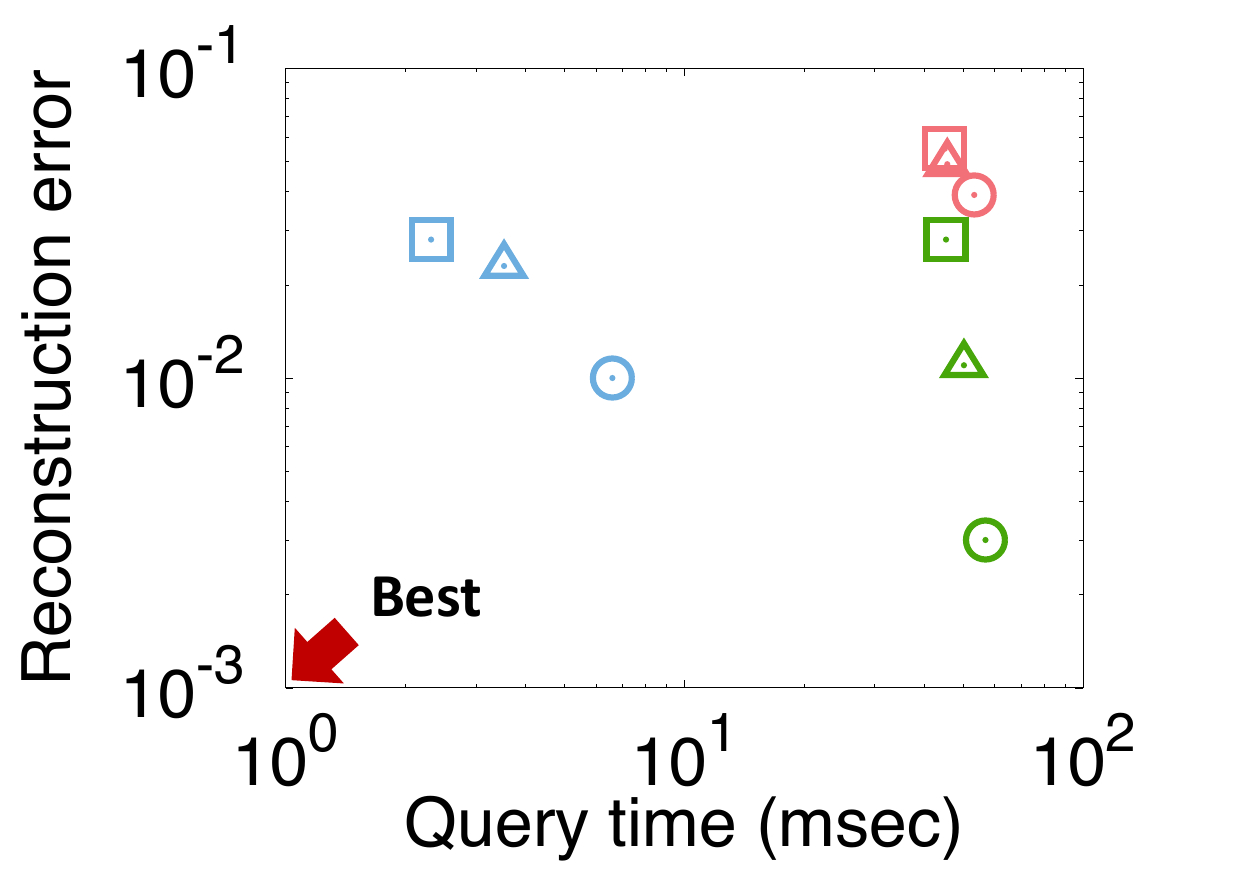}\label{fig:TRADEOFF_AT3}} \\
	\subfloat[Memory usage vs Error in Activity]{\includegraphics[ width=0.25\textwidth]{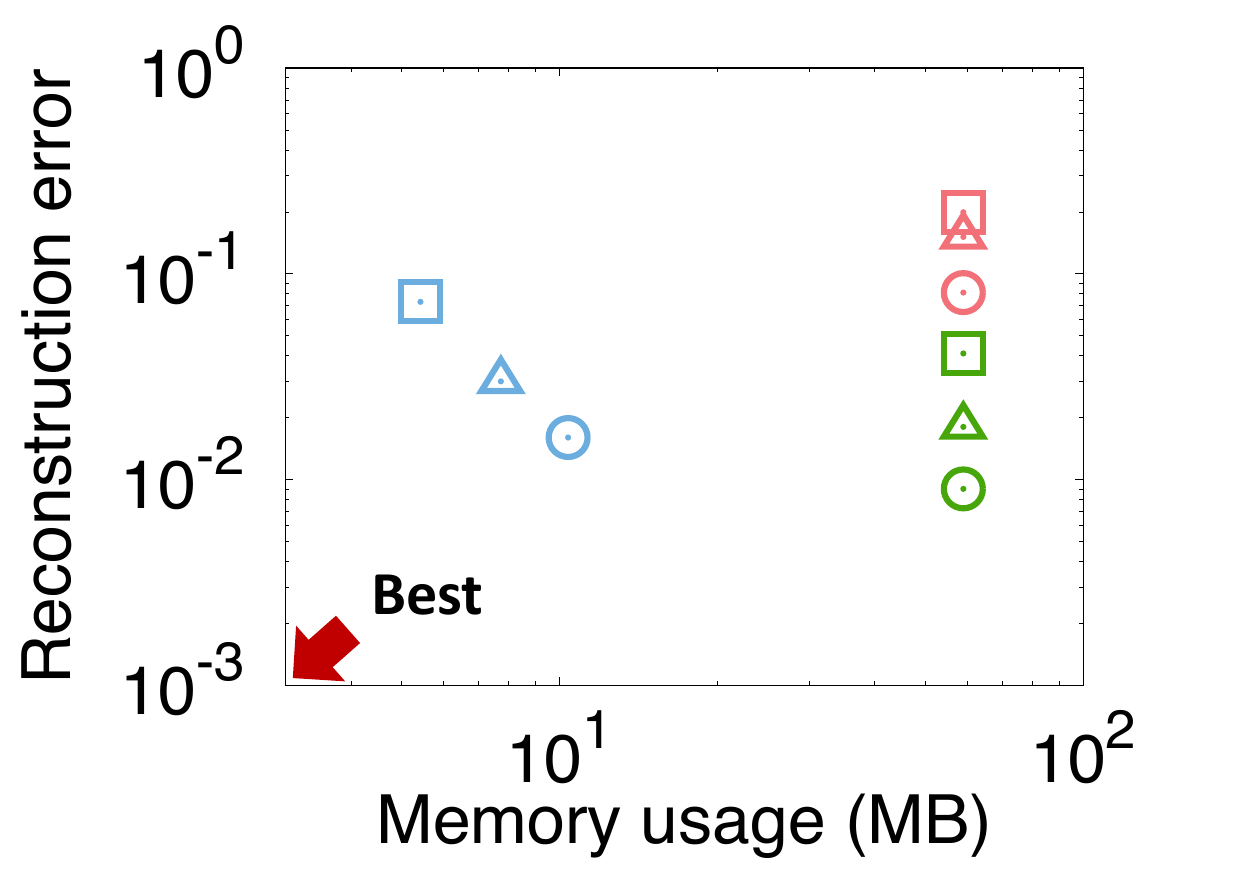}\label{fig:TRADEOFF_AM1}}
	\subfloat[Memory usage vs Error in Gas]{\includegraphics[ width=0.25\textwidth]{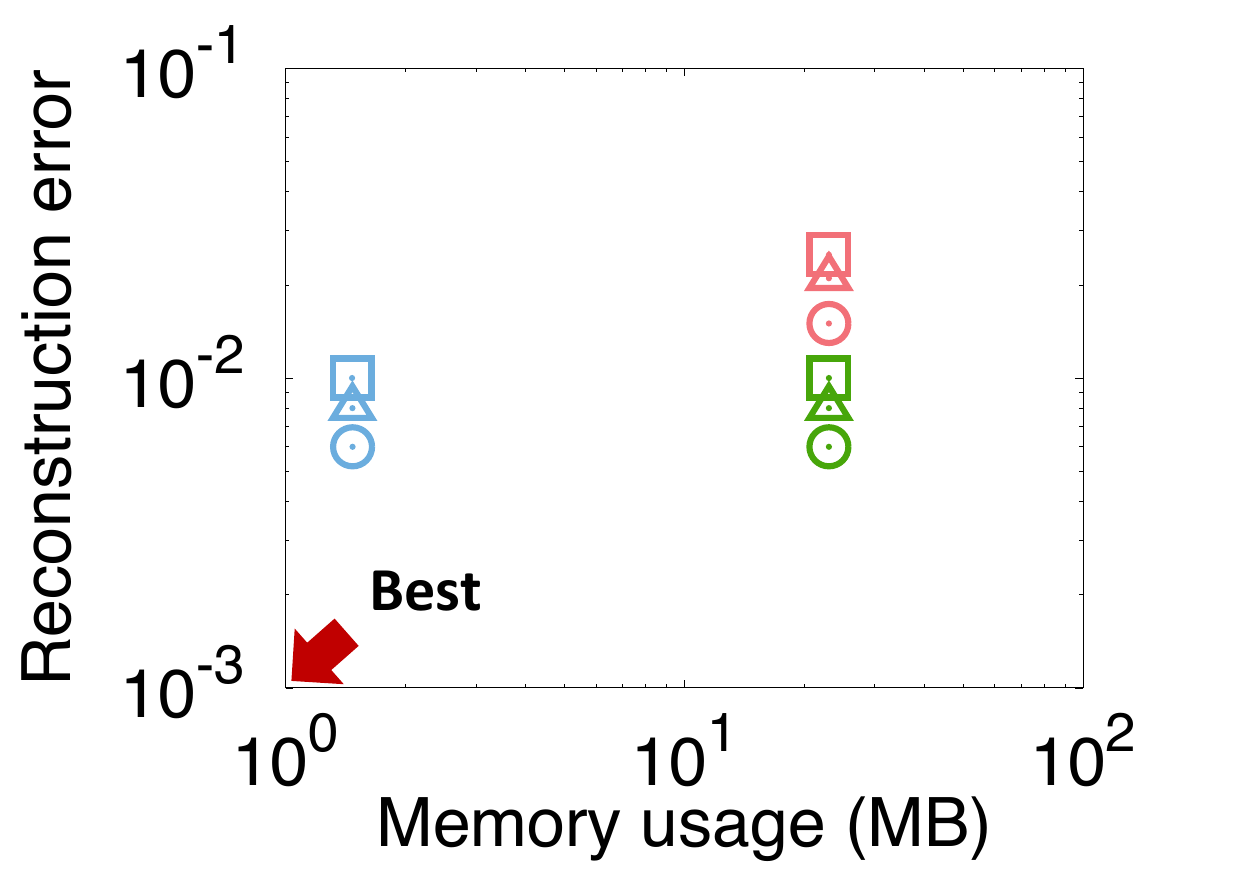}\label{fig:TRADEOFF_AM2}}
	\subfloat[Memory usage vs Error in London]{\includegraphics[ width=0.25\textwidth]{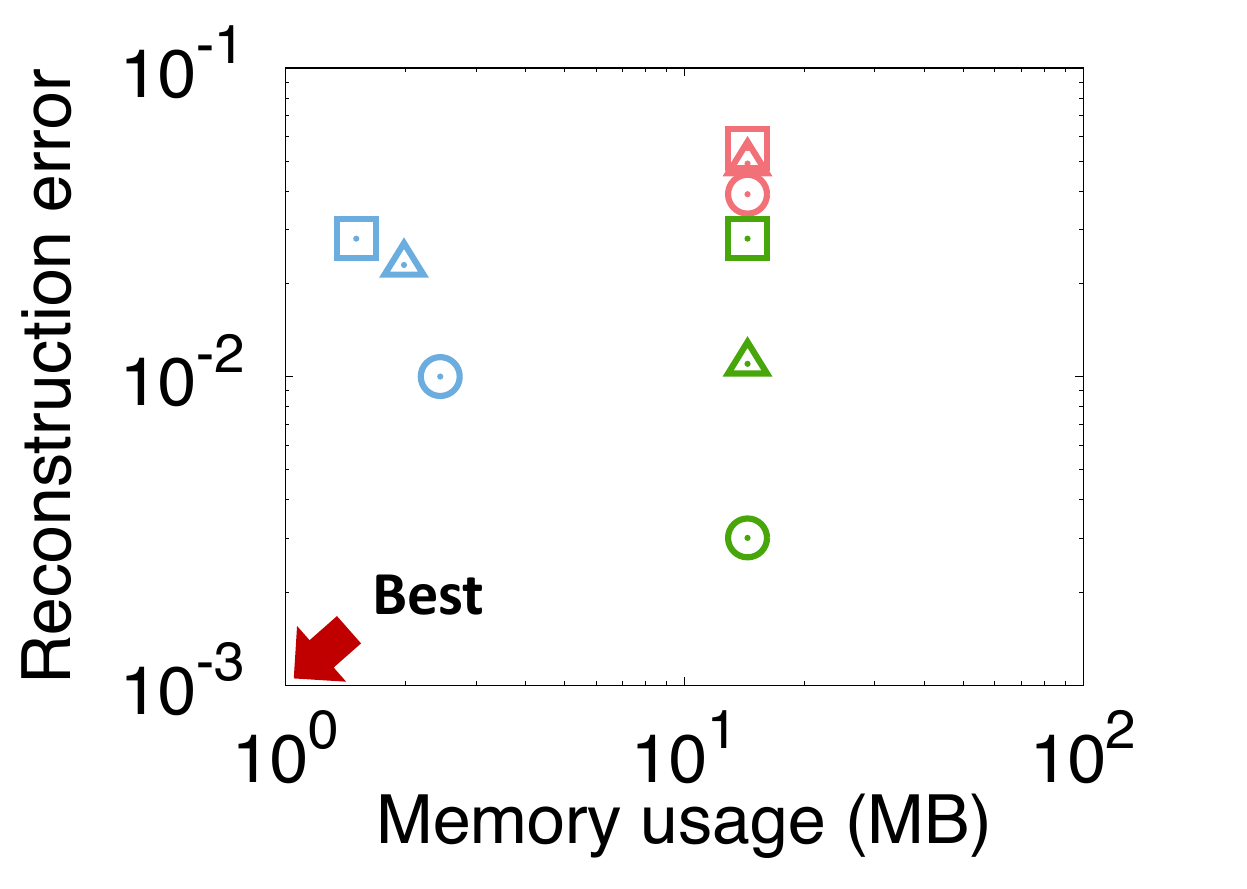}\label{fig:TRADEOFF_AM3}} \\
	%\vspace{-4mm}
	\caption{The trade-off between query time, space, and  MSE rate on three real-world datasets. The first column shows the results on the Activity dataset. The second and third columns show the results on the Gas and the London datasets. The colors represent the methods, and the shapes distinguish the threshold $\xi$. The bottom-left region indicates the better performance. On all the real-world datasets, \method has the best performance.
		%		%Bottom-left region indicates the best performance.
		%		We measure those as the threshold value $\xi$ set to $0.95$, $0.98$, and $0.99$.
		%		The values measured by \method are much closer to the bottom-left region than Naive SVD does, giving smaller error, memory usage, and running time.
	}
	\label{fig:tradeoff}
\end{figure*}
\hspace{3mm}\textbf{Methods.}
We compare \method with Randomized SVD~\cite{DBLP:journals/siamrev/HalkoMT11}, Tall and Skinny SVD~\cite{DBLP:conf/kdd/ZadehMUYPVSSZ16}, and basic SVD of JBLAS library.
All these methods are implemented using JAVA, and we use JBLAS, an open source JAVA basic linear algebra package, to support matrix operations.
%{
%In this experiment, we use SVD of JAMA library in \stitchedsvd.
%We can replace Naive SVD with other SVD methods such as~\cite{DBLP:conf/kdd/ZadehMUYPVSSZ16,DBLP:journals/siamrev/HalkoMT11,DBLP:journals/siammax/IwenO16} depending on which SVD method is used in \stitchedsvd.
%}

\textbf{Dataset.}
We use multiple time series datasets \cite{DBLP:conf/iswc/ReissS12, fonollosa2015reservoir} described in Table~\ref{tab:Description}.
{Activity dataset contains data with timestamps, and 41 measured values such as heartbeat, and inertial measurement units (IMU) data installed in the hands, chest, and ankle.}
%, and generates acceleration, temperature and gyroscope data in X, Y, and Z directions.
%We use the time series data generated by 41 sensors in Activity dataset.
Gas dataset consists of timestamps, and measurement of 16 chemical sensors having 4 different types.
Activity and Gas datasets are obtained from UCI repository~\cite{Dua:2017}.
London dataset consists of timestamps, and attributes related to London air quality such as nitric oxide, temperature, and so on.
		
\textbf{Parameters}. % 실험별로 쓰인 파라미터로 분류하여 재작성
For all experiments except the ones in Section~\ref{subsec:param_sense}, we set the block size $b$ to $1000$ in the storage phase.
We evaluate the effects of the block size $b$ in Section~\ref{subsec:param_sense}.
We set the rank $k$ using Equation~\eqref{eqn:threshold} with $\xi = 0.98$ in the storage phase.

%\vspace{-0.8mm}
\subsection{Time Cost}
%\vspace{-0.8mm}
\label{subsec:time_cost}
We examine the time costs of the storage and the query phases of our proposed method \method.

\textbf{Storage phase (Algorithm~\ref{alg:storage_phase}).}
%We examine the time cost of TR-SVDstorage phase with respect to storage phase and query phase.
We measure the running time of the storage phase for multiple time series data varying the time length.
%We perform the storage phase for the time series data from time $1$ to time $t$; thus, the time length is $t$.
As shown in Figure~\ref{fig:StoringTime},
the running time of \method's storage phase is linearly proportional to the time length over all datasets.
The reason is that the storage phase processes the time series data row by row, and the time cost of processing a row is constant as discussed in Theorem~\ref{theo:time of store}; thus, the total time cost mainly depends on the time length.
%{\color{blue}
%The pattern on London dataset is similar to that on other datasets
%}
%The pattern on London dataset is the same as that on other datasets, so we omit the result.
%\begin{figure*} [t]
%	\centering
%	\subfloat[Reconstruction error of the Activity dataset]{\includegraphics[width=0.32\textwidth]{FIG/FINAL/RECONSTRUCTION/RECON_ACTIVITY}\label{fig:Recon_ACTIVITY}}
%	\subfloat[Reconstruction error of Gas sensor dataset]{\includegraphics[width=0.32\textwidth]{FIG/FINAL/RECONSTRUCTION/RECON_HTSENSOR}\label{fig:Recon_GAS}}
%	\subfloat[Reconstruction error of London air quality dataset]{\includegraphics [width=0.32\textwidth] {FIG/FINAL/RECONSTRUCTION/RECON_LONDON.pdf}\label{fig:Recon_London}} \\
%	\caption{{Normalized reconstruction error computed by the original and the reconstructed data.
%	Reconstructed data is obtained by multiplying the stored SVD results from \method.
%	For Activity and Gas sensor datasets, the reconstruction error is stable
%	while the reconstruction error of London dataset fluctuates due to missing values.}
%	}
%	\label{fig:reconstruction2}
%\end{figure*}

\textbf{Query phase (Algorithm~\ref{alg:opt_query_phase}).}
We evaluate the performance of the query phase in terms of running time.
We compare \method to other SVD based methods discussed in Section~\ref{sec:prelim}, and measure the query time varying the length of the query range.
The starting and the ending points of the query are arbitrarily chosen, and we increase the length of the query range from $10^4$ to $3.2 \times 10^5$.
As shown in Figure~\ref{fig:Runtime}, \method is up to $9.6 \times$ faster than the second best competitor Randomized SVD on Activity dataset.
Also, \method shows up to $15 \times$ and $12 \times$ faster query speed than Randomized SVD on Gas sensor and London datasets, respectively.

%\textbf{Optimization (Algorithms~\ref{alg:opt_query_phase}).}
%We evaluate the performance of optimization when we compute the SVD with respect to an arbitrary time range in query phase.
%We choose the starting and ending points of the time range query randomly, and we increase the length of the query range from $10^4$ to $2 \times 10^5$ by step size $5 \times 10^3$.
%As shown in Figure~\ref{fig:optimization}, optimized \method is up to $39 \times$ faster than that basic \method for the Activity dataset. Also, the optimization improves the query speed by up to $63 \times$ and $54 \times$ on Gas sensor, and London datasets, respectively.
%
%\begin{figure}[t]
%	\subfloat[Reconstruction error versus space]{\includegraphics[width=0.235\textwidth]{FIG/FINAL/TRADEOFF/TRADEOFF_AM.pdf}\label{fig:tradeoff_am}}
%	\hspace{-5mm}\subfloat[Query time versus space]{\includegraphics[ width=0.235\textwidth]{FIG/FINAL/TRADEOFF/TRADEOFF_MQ.pdf}\label{fig:tradeoff_mq}}  \\
%	\caption{Trade-off between query time, space, and  error in Activity dataset.
%		%Bottom-left region indicates the best performance.
%		We measure those as the threshold value $\xi$ set to $0.95$, $0.98$, and $0.99$.
%		The values measured by \method are much closer to the bottom-left region than Naive SVD does, giving smaller error, memory usage, and running time.
%	}
%	\label{fig:tradeoff}
%\end{figure}

%\vspace{-1mm}

\subsection{Space Cost}
%\vspace{-0.8mm}
\label{subsec:space_cost}
%We examine the space cost and the scalability of \method.
%We investigate the performance of our method \method in terms space.

%\textbf{Compression performance of the storage phase.}
We evaluate the compression performance of \method.
%We store the SVD results of multiple time series data with block size $b=10^3$, and apply low-rank approximation with threshold $\xi = 0.98$  in Equation~\eqref{eqn:threshold}.
In the storage phase, we store the SVD results of multiple time series data with block size $b=10^3$ using low-rank approximation with threshold $\xi = 0.98$  in Equation~\eqref{eqn:threshold}.
We measure the compression ratio for storing the original data and the block compressed data (i.e., the SVD results) by our method.
Figure~\ref{fig:space_cost} shows the compression performance for storing time series data.
\method requires up to $7.88 \times$, $15.63 \times$, and $7.19 \times$ less space than the original data require for Activity, Gas, and London, respectively.
\subsection{Trade-off between Accuracy and Efficiency}
\label{subsec:tradeoff}
We evaluate the trade-off of \method between accuracy, time, and space compared to other methods.
{We measure the time and the space usage by setting the length of time range query to $1.8\times 10^5$.}
The accuracy of each method is measured by the reconstruction error $\frac{\lVert \mathbf{X}_{(t_s:t_e)} - \hat{\mathbf{X}}_{(t_s:t_e)} \rVert^2_{F}}{\lVert \mathbf{X}_{(t_s:t_e)} \rVert^2_{F}}$ where $\hat{\mathbf{X}}_{(t_s:t_e)}$ is the reconstructed data (i.e., $\hat{\mathbf{X}}_{(t_s:t_e)}=\Umat{S,E}\Smat{S,E}\Vmat{S,E}^T$), and $\mathbf{X}_{(t_s:t_e)}$ is the original input data in a query time range $[t_s, t_e]$.
We use the values $0.95$, $0.98$, and $0.99$ for the threshold $\xi$ in the low-rank approximation of each method to investigate the effect of $\xi$ on the trade-off performance.
Figure~\ref{fig:tradeoff} demonstrates the experimental results on the trade-off of SVD based methods including \method.
%that \method provides a better trade-off between reconstruction error and time (or space) than other competitors.
Figures~\ref{fig:TRADEOFF_AT1},~\ref{fig:TRADEOFF_AT2},~and~\ref{fig:TRADEOFF_AT3} present the trade-off of \method between query time and reconstruction error is better than those of other methods over all the datasets.
Also, \method shows a better trade-off between space and reconstruction error than its competitors as shown in Figures~\ref{fig:TRADEOFF_AM1},~\ref{fig:TRADEOFF_AM2},~and~\ref{fig:TRADEOFF_AM3}.
These results indicate that \method handles time range queries more efficiently with smaller error and space than other SVD based methods.
%Figure~\ref{fig:tradeoff} shows the points corresponding to \method are closest to the best point.
%, i.e., \method handles time ranged queries more quickly with smaller error than other SVD based methods.
%\method provides a better trade-off between the query time and the reconstruction error, giving small error while using smaller amount of query time than other SVD methods do.
%In addition, \method provides a better trade-off between the space and the reconstruction error than competitors do,
%giving small error while using smaller amount of memory as shown in Figure~\ref{fig:TRADEOFF_AM1}, ~\ref{fig:TRADEOFF_AM2}, and ~\ref{fig:TRADEOFF_AM3}.
%\vspace{-1.5mm}

\subsection{Parameter Sensitivity}
%\vspace{-1.5mm}
\label{subsec:param_sense}
We examine the effects of the block size $b$ in terms of query time, and space cost.
When the block size $b$ grows from $10$ to $10^5$, we measure the query time and the memory usage in query phase by setting the query time range $t_e-t_s+1$ to $10^3$, $10^4$, and $10^5$ in Algorithm~\ref{alg:opt_query_phase}.
In Figure~\ref{fig:param}, the query time and memory usage show trade-off characteristics. The query time and the space usage for Activity dataset decrease until the block size reaches $10^2\sim10^3$, and then increase afterwards when the block size $b$ exceeds $10^3$.
Note that \method consists of \partialsvd and \stitchedsvd modules in the query phase.
In Figure~\ref{fig:PARAM_TIME}, the query time is dominated by the computation time of the \stitchedsvd when the block size $b$ is relatively small.
{The reason is that the computation time of \partialsvd decreases as the block size $b$ decreases, while that of \stitchedsvd increases linearly with the number of blocks which is inversely proportional to the block size.}
On the other hand, the query time highly depends on the computation time of \partialsvd when the block size is relatively large.
In Figure~\ref{fig:PARAM_MEMORY}, space cost is high when $b$ is small because the number of stored singular vector matrices for blocks increases.
On the other hand, space cost is also high when $b$ is large because the number $k$ of singular values increases as the block size $b$ increases, and then the size of left singular vector matrices $\Umat{i}$ increases.
%For example, the size of left singular vector matrix $\Umat{i}' \in \mathbb{R}^{b' \times k'}$ is larger than the sum of the size of $\Umat{i}$ $\in \mathbb{R}^{b \times k_1}$ and $\Umat{i+1}$ $\in \mathbb{R}^{b \times k_2}$ when $k' > k_1+k_2$ and $b' = 2b$.
Therefore, we set the block size $b$ to $10^3$ which is a near-optimum value according to the experimental results.
The sensitivities of \method on the block size $b$ in other datasets show similar patterns.
\begin{figure} [t]
	%\vspace{-3mm}
	\subfloat[Query time vs. block size]{\includegraphics[width=0.235\textwidth]{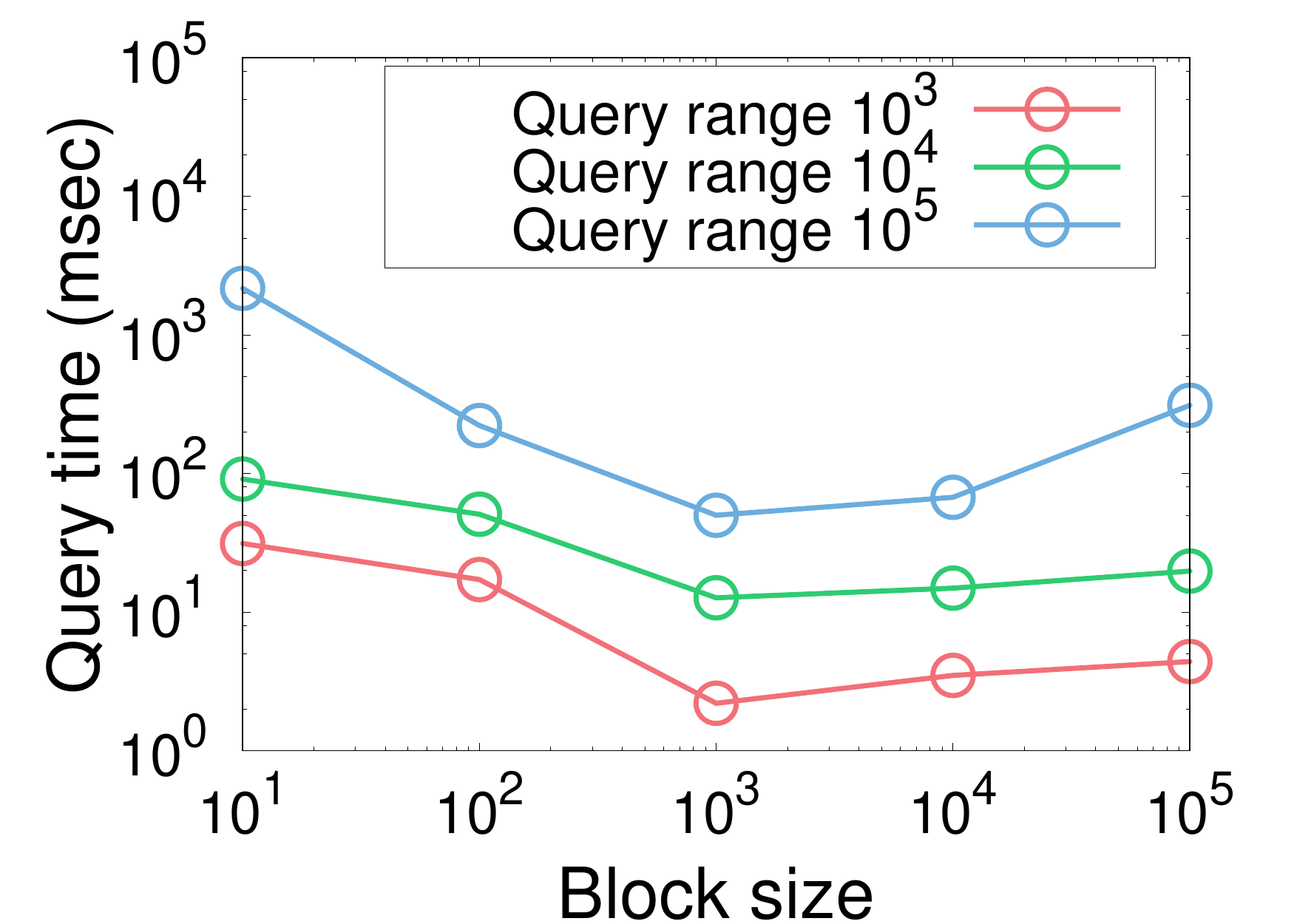}\label{fig:PARAM_TIME}}
	\hspace{-2mm}\subfloat[Space cost vs. block size]{\includegraphics[ width=0.235\textwidth]{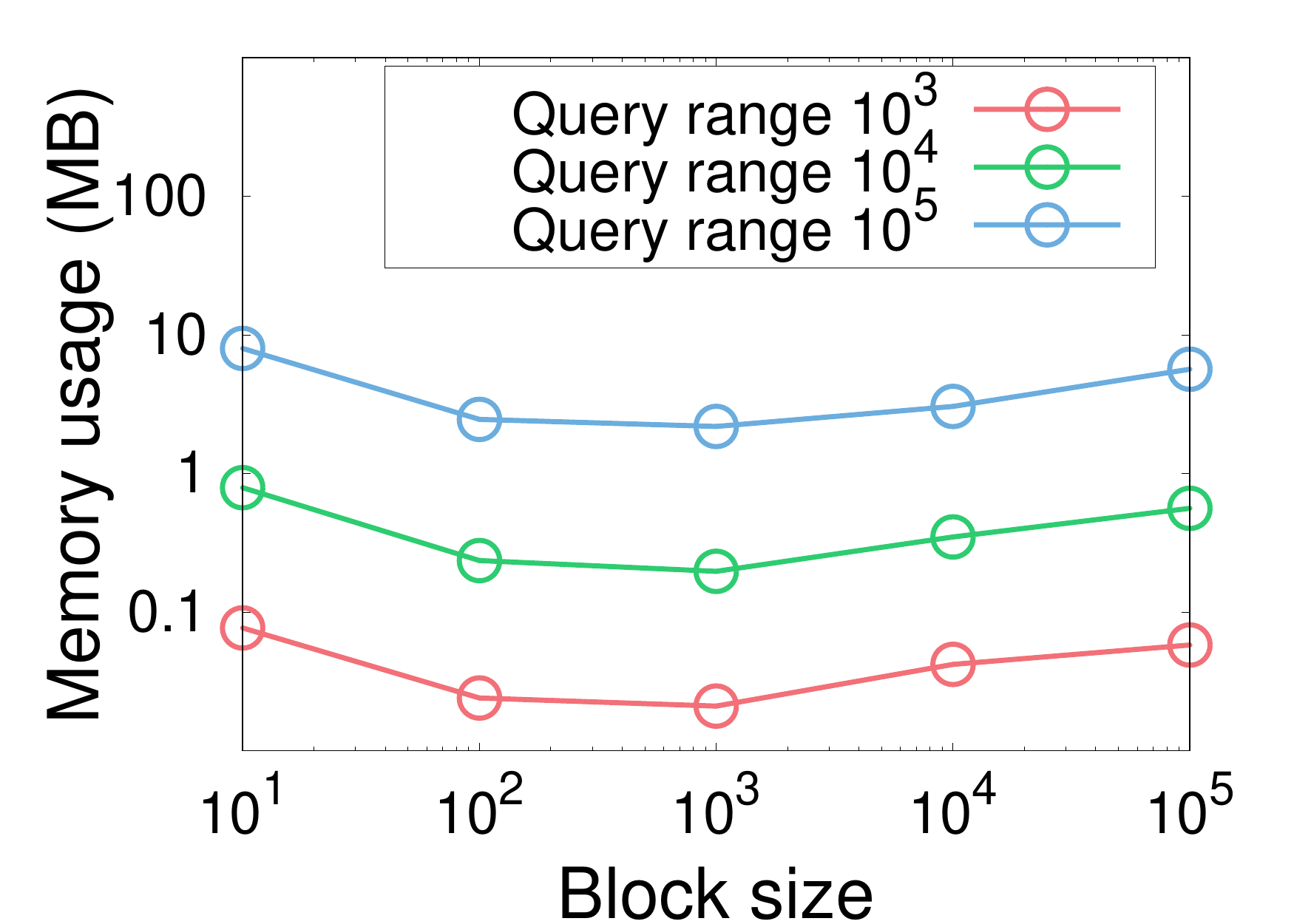}\label{fig:PARAM_MEMORY}}  \\
	\caption{The effect of the block size $b$ in terms of the query time and the memory usage in Activity dataset.
		When the block size $b$ grows from $10$ to $10^5$,
		the query time and the memory usage decrease until the block size is $10^3$, and then increase continuously after block size exceeds $10^3$.
	}
	\label{fig:param}
\end{figure}

\begin{figure*} [t]
	\vspace{-3mm}
	\centering
	\subfloat[First left singular vector]{\includegraphics[ width=0.19\textwidth]{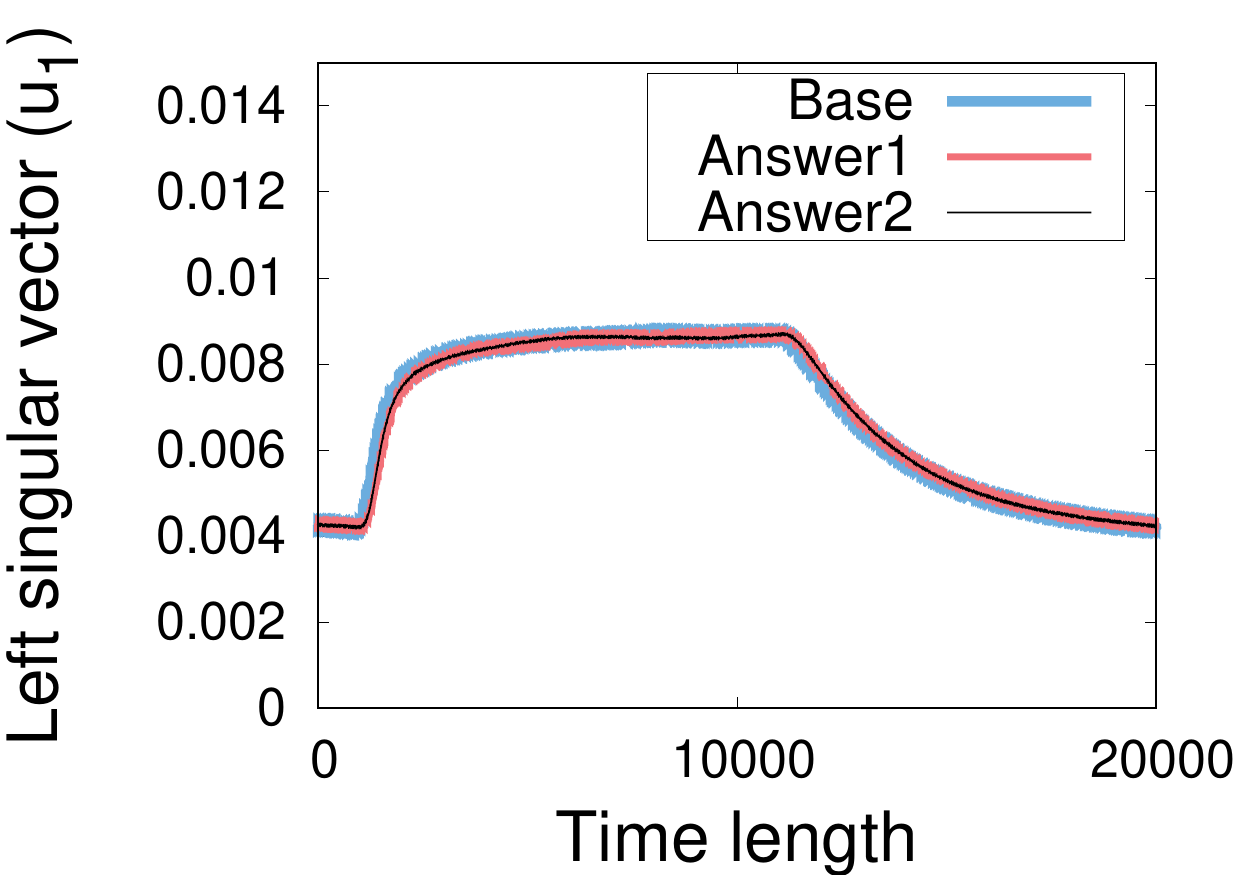}\label{fig:singular_vector}}
	\subfloat[Sensor 1]{\includegraphics[ width=0.19\textwidth]{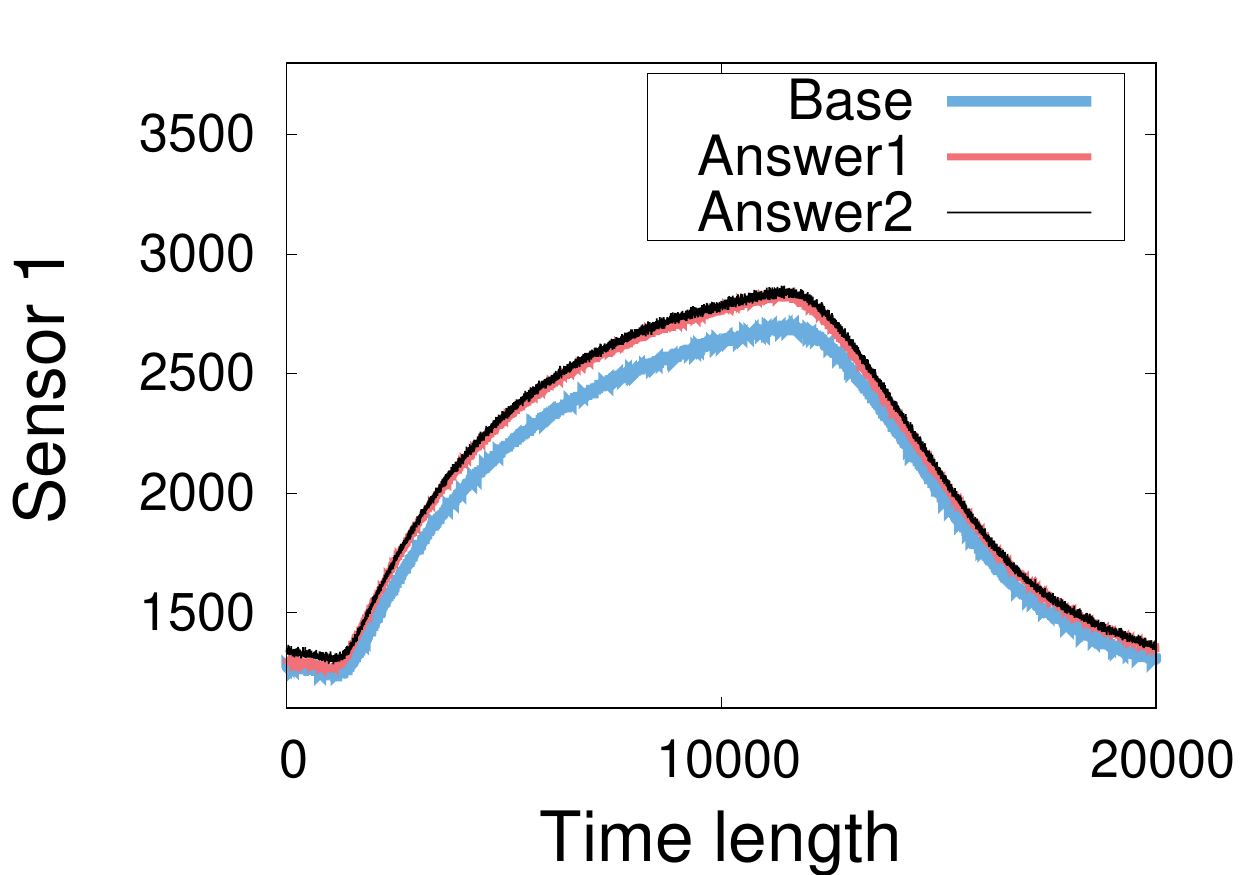}\label{fig:Sensor1}}
	\subfloat[Sensor 2]{\includegraphics[ width=0.19\textwidth]{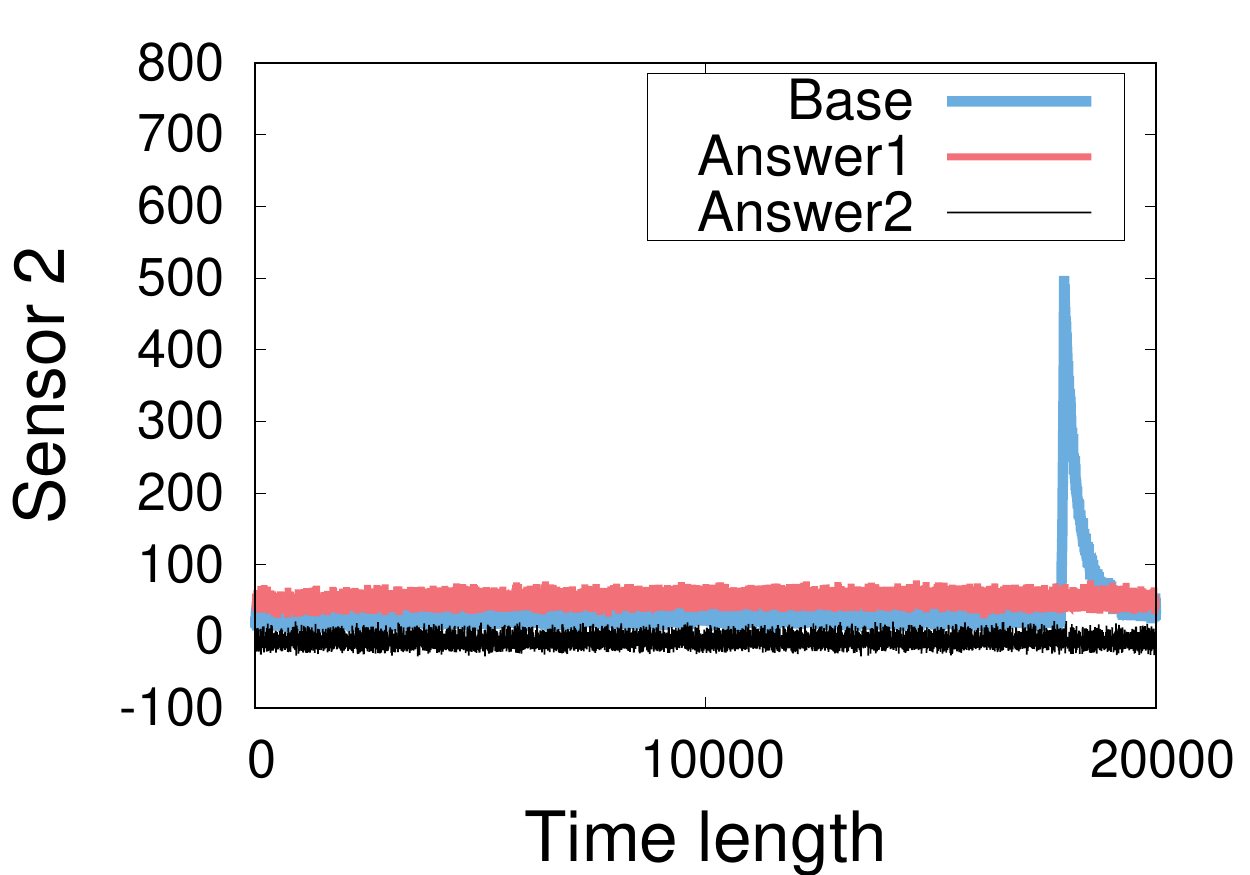}\label{fig:Sensor2}}
	\subfloat[Sensor 3]{\includegraphics[ width=0.19\textwidth]{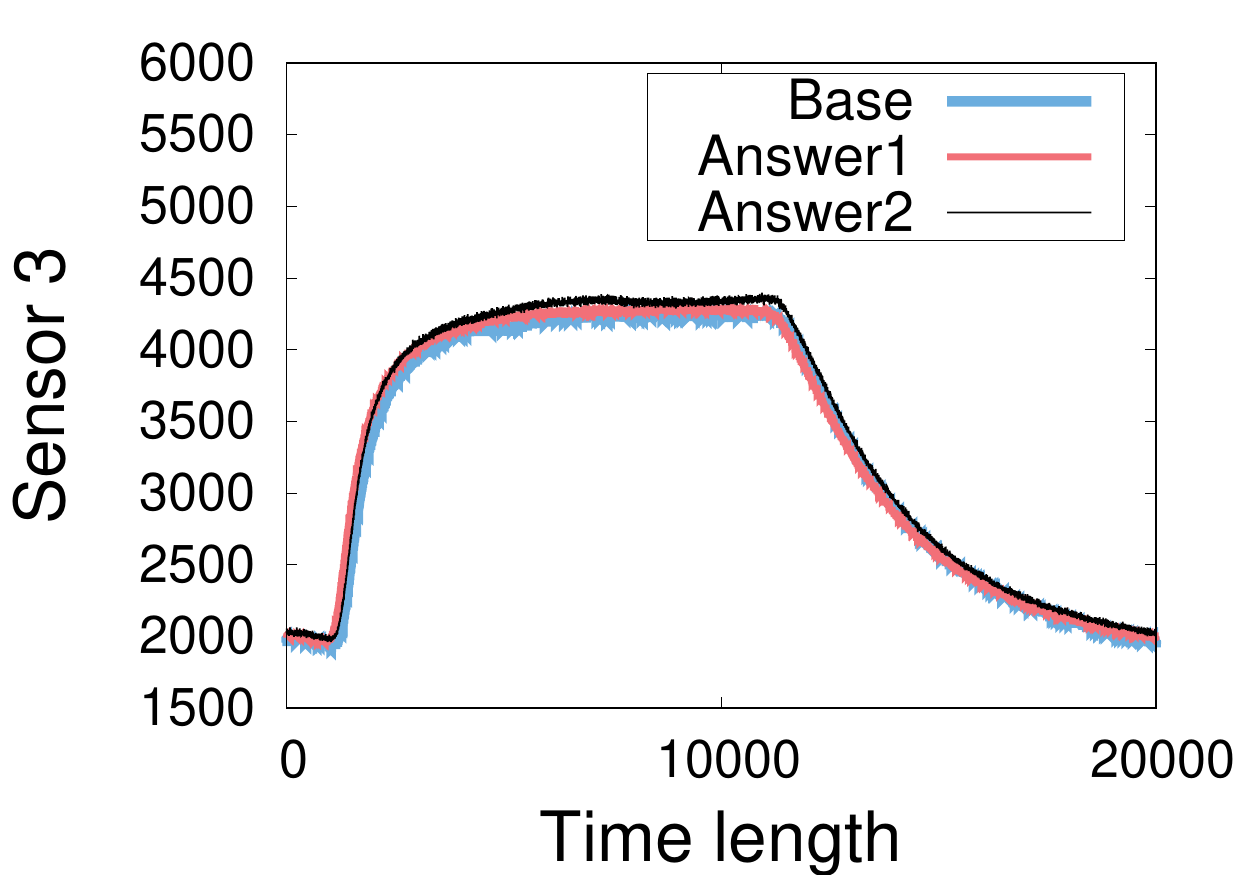}\label{fig:Sensor3}}
	\subfloat[Sensor 4]{\includegraphics[ width=0.19\textwidth]{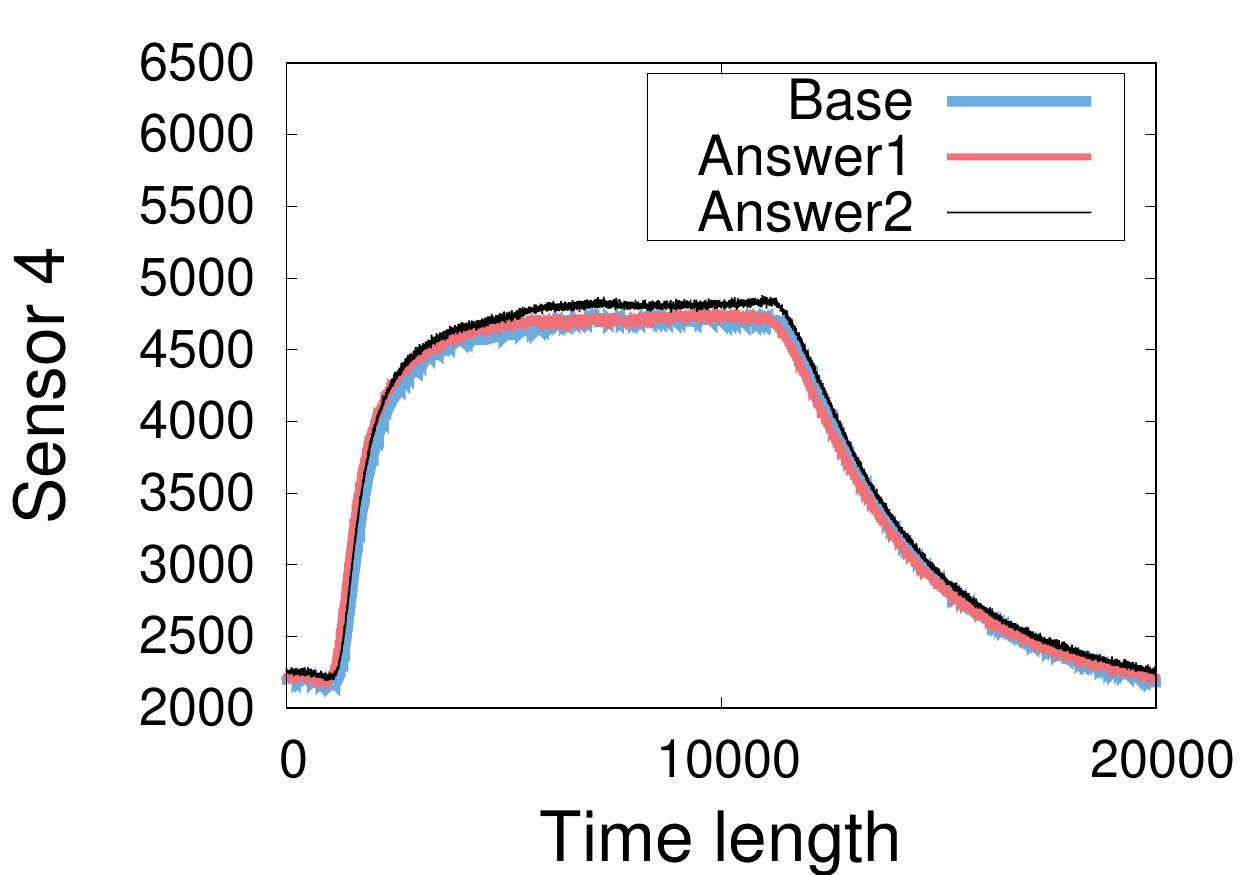}\label{fig:Sensor4}}   \\
	\caption{
		First left singular vectors of the base and two time ranges having the similar patterns to that of the base.
		Answer1 and Answer2 are the past time ranges we found using \method by sliding a window.
		The time range of Answer1 is $[1589294, 1609293]$, and that of Answer2 is $[2736794, 2756793]$.
		In (a), the three left singular vectors have the same pattern.
		In (b, c, d, e), each sensor shows similar measurements for the three ranges in general.
Note that \method discovers a potential anomaly: (c) shows an abnormal spike of Base around $17,960$, which deviates vastly from the measurements of the past ranges.
}
	\label{fig:discovery}
\end{figure*}

\section{Case Study}
\label{sec:casestudy}
% !TeX spellcheck = en_US
In this section, we present the analysis result of Gas sensor dataset using \method.
For a given time range, \method
%(a) extraction of key trends, (b) clustering, and (c)
searches past time ranges for similar patterns to that of the given time range.
%We also show the other case studies such as extraction of key trends and clustering in~\cite{Supplementary}.}
%For discoveries, we describe Gas sensor dataset in more details as follows.

\textbf{Data description.}
Gas sensor dataset consists of 16 chemical sensors exposed to gas mixtures which contain ethylene and carbon monoxide (CO) with air.
There are four different types of sensors: TGS2600, TGS2602, TGS2610, and TGS2620 (four sensors per one type).
The sensors are placed in a 60ml measurement chamber.
The electrical conductivities of the sensors are measured at every 0.01 seconds, and the total time length is 12 hours.
%\vspace{-2mm}

\textbf{Finding similar time ranges.}
Our goal is to find past time ranges whose patterns are similar to that of a given time range query which we call as the \lq base\rq.
Suppose a time range $[t_s, t_e]$ is given.
First, we compute SVD of data in the range using \method. %; we call this as the base SVD.
Next, we continuously compute SVDs of previous time ranges using \method by sliding a window of size $t_e-t_s+1$. % which captures a candidate answer time range.
Note that we use the fixed window size for efficiency.
%The window slides with a sliding period for search efficiency.
We then compare $\mathbf{u_1}$ of the base with $\mathbf{\overline{u}_1}$ of previous time ranges, where $\mathbf{u_1}$ is the first column of left singular vector matrix $\mathbf{U}$ of the base, and $\mathbf{\overline{u}_1}$ is the first column of $\mathbf{\overline{U}}$ of a previous time range.
Experimental setting is as follows:
\begin{itemize}
%	\item We use Gas sensor dataset.
	\item  {The base time range is $[4033294, 4053293]$ which is randomly chosen, and thus the size of sliding window (time length) is $20,000$. The sliding period is $500$.}
	\item We compute the cosine similarity $\frac{\mathbf{u_1} \cdot \mathbf{\overline{u}_1}}{\lVert \mathbf{u_1} \rVert \lVert \mathbf{\overline{u}_1} \rVert}$ to compare the patterns, and then find the two most similar time ranges.
\end{itemize}
Figure~\ref{fig:discovery} shows the first singular vectors of the given time range (denoted by 'Base') and two previous time ranges (denoted by 'Answer1' and 'Answer2') with similar patterns we found, and sensor data corresponding to each time range.
In Figure~\ref{fig:singular_vector}, \method successfully finds the similar patterns, and each sensor has the same tendency w.r.t. the three time ranges. % as those have the same key trend.
We also discover a potential anomaly.
In Figure~\ref{fig:Sensor2}, note an abnormal spike of Base at the time around $17,960$, which deviates significantly from the measurements of the two previous time ranges. 

\section{Related Work}
\label{sec:related}

In this section, we discuss related works on incremental time series analysis and incremental SVD.

%\textbf{Incremental time series analysis.}
It is important to analyze on-line time series data efficiently. There are several ideas based on linear modeling, including
Kalman Filters (KF), linear dynamical systems (LDS) and the variants~\cite{jain2004adaptive,li2010parsimonious,li2009dynammo}, time warping~\cite{toyoda2013pattern}, and correlation-based methods~\cite{mueen2010fast,sakurai2005braid,zhu2002statstream,DBLP:conf/sigmod/KeoghCMP01}. Among them,
incremental SVD efficiently tracks the SVD of dynamic time series data where new data rows are incrementally attached to the current data.
Incremental SVD has been used for updating SVD results when new terms or documents are added to a document-term matrix~\cite{deerwester1990indexing}, and has been adopted to build an incremental movie recommender system~\cite{sarwar2002incremental,brand2003fast}.
Brand~\cite{brand2002incremental} developed an incremental SVD method for incomplete data which contain missing values.
Papadimitriou et al.~\cite{DBLP:conf/vldb/PapadimitriouSF05} proposed SPIRIT which incrementally updates hidden factors in multiple time series, and exploited SPIRIT to predict future signals and interpolate missing values in sensor streams.
Ross et al.~\cite{ross2008incremental} proposed an incremental PCA for visual tracking applications.
Papadimitriou et al.~\cite{DBLP:conf/sigmod/PapadimitriouY06} introduced an incremental method to capture optimal recurring patterns indicating the main trends in time series data.
%TODO compare the above methods to the proposed method

The methods above have limitations in applying to time range query problem. Conventional methods { such as Discrete Fourier Transformation~\cite{mueen2010fast}, data compression~\cite{reeves2009managing,gandhi2009gamps}, and time series clustering~\cite{paparrizos2015k}} require the whole time series data, or additional data to seek hidden patterns for a specific time range, thereby incurring high memory and time complexity.
On the other hand, our proposed \method efficiently serves time range queries with small memory and low time complexity.

%Incremental SVD needs additional processing for time-varying query problem because it does not provide resultant factors over a certain time range. 

\section{Conclusions}
\label{sec:conclusion}
We propose \method, a novel algorithm for
finding key patterns in an arbitrary time range from multiple time series data.
\method efficiently serves time range queries by compressing multiple time series data, and reducing computation costs by carefully stitching compressed SVD results.
Consequently, \method provides fast query time and small memory usage.
We provide theoretical analysis on the time and space complexities of \method.
Experiments show that \method requires up to $15.66 \times$ less space, and runs up to $15 \times$ faster than existing methods.
%Also, we show real-world case studies of \method for finding key trends, clustering, and searching past time ranges having similar patterns as that of a query time range.
Also, we show a real-world case study of \method
in searching past time ranges for similar patterns as that of a query time range.
Future research includes extending the method for multiple distributed streams. 

\begin{acks}
This work was supported by the National Research Foundation of Korea(NRF) funded by the Ministry of Science, ICT and Future Planning (NRF-2016M3C4A7952587, PF Class Heterogeneous High Performance Computer Development).
The Institute of Engineering Research at Seoul National University provided research facilities for this work.
The ICT at Seoul National University provides research facilities for this study.
U Kang is the corresponding author.
\end{acks}

\balance
\newpage

\bibliographystyle{ACM-Reference-Format}
\bibliography{mybib}

\end{document}